\documentclass[11pt,oneside,english]{amsart}
\usepackage[T1]{fontenc}
\usepackage[latin9]{inputenc}
\usepackage{geometry}
\usepackage{color}
\usepackage{units}
\usepackage{url}
\usepackage{amsbsy}
\usepackage{amstext}
\usepackage{amsthm}
\usepackage{amssymb}

\makeatletter

\newcommand{\noun}[1]{\textsc{#1}}
\newcommand{\mathcircumflex}[0]{\mbox{\^{}}}

\numberwithin{equation}{section}
\numberwithin{figure}{section}
\theoremstyle{plain}
\newtheorem{thm}{\protect\theoremname}
\theoremstyle{plain}
\newtheorem{conjecture}[thm]{\protect\conjecturename}
\theoremstyle{plain}
\newtheorem{prop}[thm]{\protect\propositionname}

\usepackage{nicefrac}
\usepackage{sagetex}
\usepackage{babel}
\@ifundefined{definecolor}
 {\usepackage{color}}{}
\@ifundefined{definecolor}{\usepackage{color}}{}
\usepackage{graphicx}

\usepackage[mathscr]{euscript}
 \let\mathscr\relax
\usepackage[scr]{rsfso}
\newcommand{\powerset}{\raisebox{.15\baselineskip}{\Large\ensuremath{\wp}}}

\usepackage{mathtools}

\makeatother

\usepackage{babel}
\providecommand{\conjecturename}{Conjecture}
\providecommand{\propositionname}{Proposition}
\providecommand{\theoremname}{Theorem}

\begin{document}

\title{Sketch for a Theory of Constructs}

\author{Edinah K. Gnang, Jeanine Gnang}
\begin{abstract}
The present note sketches a theory of constructs.
\end{abstract}

\maketitle

\section{What are constructs ?}

The many variations around the general theme of matrix multiplication
\cite{MB94,GKZ,RK,GER,MB90,Lim2013,Zwick:2002:PSP:567112.567114,crane2015,BUTKOVIC2003313,5387448,Iverson:1962:PL:1098666}
press for a unified framework. A partial unification is achieved by
adopting an insight from category theory. Namely, the observation
that composition monoids broaden the scope of the multiplication operation.
The basic idea of the proposed theory of \emph{constructs} is to assign
to matrix and hypermatrix entries, morphisms of a semi-category (for
which the associativity requirement as well as properties of the identity
element are loosened). We call the resulting mathematical objects
\emph{constructs. }The algebra of \emph{constructs} is prescribed
by a \emph{combinator} noted Op, and a \emph{composer} noted $\mathcal{F}$.
The \emph{composer} specifies rules for composing entry morphisms
while the \emph{combinator} specifies rules for combining the compositions
of entry morphisms. Natural choices for a \emph{combinator} include
:
\[
\sum_{0\le\textcolor{red}{j}<\ell},\:\prod_{0\le\textcolor{red}{j}<\ell},\:\max_{0\le\textcolor{red}{j}<\ell},\:\min_{0\le\textcolor{red}{j}<\ell},\:\bigcup_{0\le\textcolor{red}{j}<\ell},\:\bigcap_{0\le\textcolor{red}{j}<\ell},\:\bigtimes_{0\le\textcolor{red}{j}<\ell},\bigoplus_{0\le\textcolor{red}{j}<\ell},\,\bigotimes_{0\le\textcolor{red}{j}<\ell},\,\bigvee_{0\le\textcolor{red}{j}<\ell},\,\bigwedge_{0\le\textcolor{red}{j}<\ell}
\]
respectively associated with the summation, the product, the maximum,
the minimum, the union, the intersection, the cartesian product, the
direct sum, the tensor product, the boolean disjunction and the boolean
conjunction. For instance, the product of second-order constructs
$\mathbf{A}$ and $\mathbf{B}$ of size respectively $m\times\textcolor{red}{\ell}$
and $\textcolor{red}{\ell}\times n$ results in a construct $\mathbf{C}$
of size $m\times n$ noted
\[
\mathbf{C}=\text{CProd}{}_{\text{Op},\mathcal{F}}\left(\mathbf{A},\mathbf{B}\right).
\]
The entries of $\mathbf{C}$ \footnote{For the readers convenience, the present note is interspersed with
illustrative SageMath \cite{sage} code snippets which use the Hypermatrix
Algebra Package available at the link \url{https://github.com/gnang/HypermatrixAlgebraPackage}} are specified by 
\[
\mathbf{C}\left[i,j\right]=\begin{array}{c}
\\
\text{Op}\\
^{0\le\textcolor{red}{t}<\ell}
\end{array}\mathcal{F}\left(\mathbf{A}\left[i,\textcolor{red}{t}\right],\,\mathbf{B}\left[\textcolor{red}{t},j\right]\right),\ \forall\,\begin{cases}
\begin{array}{c}
0\le i<m\\
0\le j<n
\end{array}.\end{cases}
\]
More specifically, the product of $2\times2$ constructs is given
by 
\[
\mbox{CProd}_{\text{Op},\mathcal{F}}\left(\left(\begin{array}{cc}
a_{00} & a_{01}\\
a_{10} & a_{11}
\end{array}\right),\left(\begin{array}{cc}
b_{00} & b_{01}\\
b_{10} & b_{11}
\end{array}\right)\right)=\left(\begin{array}{ccc}
\text{Op}\left(\mathcal{F}\left(a_{00},b_{00}\right),\mathcal{F}\left(a_{01},b_{10}\right)\right) &  & \text{Op}\left(\mathcal{F}\left(a_{00},b_{01}\right),\mathcal{F}\left(a_{01},b_{11}\right)\right)\\
\\
\text{Op}\left(\mathcal{F}\left(a_{10},b_{00}\right),\mathcal{F}\left(a_{11},b_{10}\right)\right) &  & \text{Op}\left(\mathcal{F}\left(a_{10},b_{01}\right),\mathcal{F}\left(a_{11},b_{11}\right)\right)
\end{array}\right).
\]
The algebra of constructs therefore generalizes the algebra of matrices
and hypermatrices. In particular, the simplest way to recover the
usual matrix product from the product of second order constructs is
obtained by setting the composer and combinator to the product and
sum respectively,

\[
\mathcal{F}\left(x,y\right)\,:=x\times y\quad\text{ and }\quad\begin{array}{c}
\\
\text{Op}\\
^{0\le\textcolor{red}{t}<\ell}
\end{array}\,:=\sum_{0\le\textcolor{red}{t}<\ell}.
\]
The corresponding SageMath code setup is as follows\\
\begin{sageverbatim}
sage: # Loading the Hypermatrix Algebra Package into SageMath
sage: load('./Hypermatrix_Algebra_Package_code.sage')
sage: 
sage: # Initialization of the constructs which in this
sage: # particular case are symbolic 2x2 matrices
sage: mA=HM(2,2,'a'); mB=HM(2,2,'b')
sage: mA
[[a00, a01], [a10, a11]]
sage: mB
[[b00, b01], [b10, b11]]
sage: 
sage: # Computing the construct product where
sage: # the combinator is set to : sum
sage: # and the composer is set to : prod
sage: mC=CProd([mA, mB], sum, prod)
sage: mC
[[a00*b00 + a01*b10, a00*b01 + a01*b11], [a10*b00 + a11*b10, a10*b01 + a11*b11]]
\end{sageverbatim}\\
The code above illustrates the initialization of constructs 
\[
\mathbf{mA}\in\left(\mathbb{C}\left[a_{00},a_{10},a_{01},a_{11}\right]\right)^{2\times2},\;\mathbf{mB}\in\left(\mathbb{C}\left[b_{00},b_{10},b_{01},b_{11}\right]\right)^{2\times2},
\]
such that
\[
\mathbf{mA}=\left(\begin{array}{rr}
a_{00} & a_{01}\\
a_{10} & a_{11}
\end{array}\right),\quad\mathbf{mB}=\left(\begin{array}{rr}
b_{00} & b_{01}\\
b_{10} & b_{11}
\end{array}\right),
\]
and also illustrates the computation of the product

\begin{sagesilent}
# Loading the Hypermatrix Algebra Package into SageMath
load('./Hypermatrix_Algebra_Package_code.sage')

# Initialization of the constructs which in this
# particular case are symbolic 2x2 matrices
mA=HM(2,2,'a'); mB=HM(2,2,'b')

# Computing the construct product where
# the combinator is set to : sum
# and the composer is set to : prod
mC=CProd([mA, mB], sum, prod)
\end{sagesilent}
\[
\mathbf{mC}=\mbox{CProd}_{\sum,\times}\left(\mathbf{mA},\mathbf{mB}\right)=\left(\begin{array}{ccc}
\sage{mC[0,0]} &  & \sage{mC[0,1]}\\
\\
\sage{mC[1,0]} &  & \sage{mC[1,1]}
\end{array}\right).
\]
Other less familiar variants of matrix multiplication arise as special
instances of products of constructs. For instance, H. Crane recently
introduced in \cite{crane2015} a set-valued matrix product used to
characterize the class of exchangeable Lipschitz partition processes.
The set-valued matrix product introduced in \cite{crane2015} expresses
the product of second order constructs whose composer and combinator
are respectively set intersections and set unions
\[
\mathcal{F}\left(X,Y\right)\,:=X\cap Y\quad\text{ and }\quad\begin{array}{c}
\\
\text{Op}\\
^{0\le\textcolor{red}{t}<\ell}
\end{array}\,:=\bigcup_{0\le\textcolor{red}{t}<\ell}.
\]
The SageMath code setup for multiplying set valued matrices is as
follows\\
\begin{sageverbatim}
sage: # Loading the Hypermatrix Algebra Package into SageMath
sage: load('./Hypermatrix_Algebra_Package_code.sage')
sage: 
sage: # Initialization of the construct whose
sage: # indivdual entries are set taken 
sage: # from the power set of {1, 2, 3}
sage: sA=HM([[Set([1, 2]), Set([1, 2, 3])], [Set([1]), Set([2, 3])]])
sage: sA
[[{1, 2}, {1, 2, 3}], [{1}, {2, 3}]]
sage: sB=HM([[Set([1, 2, 3]), Set([2])], [Set([1, 3]), Set([1, 3])]])
sage: sB
[[{1, 2, 3}, {2}], [{1, 3}, {1, 3}]]
sage: # Computing the construct product where
sage: # the combinator is set to : SetUnion
sage: # and the composer is set to : SetIntersection
sage: sC=CProd([sA, sB], SetUnion, SetIntersection)
[[{1, 2, 3}, {1, 2, 3}], [{1, 3}, {3}]]
\end{sageverbatim}\\
The code above illustrates the initialization of constructs $\mathbf{sA}$,
$\mathbf{sB}\in\left(\powerset\left(\left\{ 1,2,3\right\} \right)\right)^{2\times2}$
as
\[
\mathbf{sA}=\left(\begin{array}{ccc}
\left\{ 1,2\right\}  &  & \left\{ 1,2,3\right\} \\
\\
\left\{ 1\right\}  &  & \left\{ 2,3\right\} 
\end{array}\right),\quad\mathbf{sB}=\left(\begin{array}{ccc}
\left\{ 1,2,3\right\}  &  & \left\{ 2\right\} \\
\\
\left\{ 1,3\right\}  &  & \left\{ 1,3\right\} 
\end{array}\right),
\]
and illustrates the computation of the product\begin{sagesilent}
# Loading the Hypermatrix Algebra Package into SageMath
load('./Hypermatrix_Algebra_Package_code.sage')

# Initialization of the construct whose
# indivdual entries are set taken 
# from the power set of {1, 2, 3}
sA=HM([[Set([1, 2]), Set([1, 2, 3])], [Set([1]), Set([2, 3])]]) 
sB=HM([[Set([1, 2, 3]), Set([2])], [Set([1, 3]), Set([1, 3])]])

# Computing the construct product where
# the combinator is set to : SetUnion
# and the composer is set to : SetIntersection
sC=CProd([sA, sB], SetUnion, SetIntersection)
\end{sagesilent}
\[
\mathbf{sC}=\mbox{CProd}_{\bigcup,\,\cap}\left(\mathbf{sA},\mathbf{sB}\right)=\left(\begin{array}{ccc}
\sage{sC[0,0]} &  & \sage{sC[0,1]}\\
\\
\sage{sC[1,0]} &  & \sage{sC[1,1]}
\end{array}\right).
\]
In particular note that for all $\mathbf{M}\in\left(\powerset\left(\left\{ 1,2,3\right\} \right)\right)^{2\times2}$,
where $\powerset\left(\left\{ 1,2,3\right\} \right)$ denotes the
power set of $\left\{ 1,2,3\right\} $ we have 
\[
\mbox{CProd}_{\bigcup,\,\cap}\left(\left(\begin{array}{ccc}
\left\{ 1,2,3\right\}  &  & \emptyset\\
\\
\emptyset &  & \left\{ 1,2,3\right\} 
\end{array}\right),\:\mathbf{M}\right)=\mathbf{M}=\mbox{CProd}_{\bigcup,\,\cap}\left(\mathbf{M},\:\left(\begin{array}{ccc}
\left\{ 1,2,3\right\}  &  & \emptyset\\
\\
\emptyset &  & \left\{ 1,2,3\right\} 
\end{array}\right)\right).
\]
More generally, the identity element for GProd$_{\bigcup,\,\cap}$
over $\left(\powerset\left(S\right)\right)^{n\times n}$ for a given
set $S$ is the $n\times n$ construct 
\[
\left(\begin{array}{cccc}
S & \emptyset & \cdots & \emptyset\\
\emptyset & \ddots & \ddots & \vdots\\
\vdots & \ddots & \ddots & \emptyset\\
\emptyset & \cdots & \emptyset & S
\end{array}\right).
\]
It follows from De Morgan's laws that the identity element for GProd$_{\bigcap,\,\cup}$
over $\left(\powerset\left(S\right)\right)^{n\times n}$ is the $n\times n$
construct
\[
\left(\begin{array}{cccc}
\emptyset & S & \cdots & S\\
S & \ddots & \ddots & \vdots\\
\vdots & \ddots & \ddots & S\\
S & \cdots & S & \emptyset
\end{array}\right).
\]
Similarly, a product of boolean-valued matrices also arises as a special
instance of the product of second order constructs. In the setting
of boolean-valued matrices the composer and combinator are respectively
the boolean conjunction and the boolean disjunction
\[
\mathcal{F}\left(X,Y\right)\,:=X\wedge Y\quad\text{ and }\quad\begin{array}{c}
\\
\text{Op}\\
^{0\le\textcolor{red}{t}<\ell}
\end{array}\,:=\bigvee_{0\le\textcolor{red}{t}<\ell}.
\]
The SageMath code setup for multiplying boolean-valued matrices is
as follows\\
\begin{sageverbatim}
sage: # Loading the Hypermatrix Algebra Package into SageMath
sage: load('./Hypermatrix_Algebra_Package_code.sage')
sage: 
sage: # Initialization of the construct whose indivdual
sage: #  entries are boolean values in {True, False}
sage: bA=HM([[True, False], [True,  True]])
sage: bA
[[True, False], [True, True]]
sage: bB=HM([[False, True], [True, False]])
sage: bB
[[False, True], [True, False]]
sage: # Computing the construct product where
sage: # the combinator is set to : boolean disjunction
sage: # and the composer is set to : boolean conjunction
sage: bC=CProd([bA, bB], Or, And)
sage: bC
[[False, True], [True, True]]
\end{sageverbatim}\\
The code above illustrates the initialization of constructs $\mathbf{bA}$,
$\mathbf{bB}\in\left(\left\{ \text{True},\,\text{False}\right\} \right)^{2\times2}$
as\begin{sagesilent}
# Loading the Hypermatrix Algebra Package into SageMath
load('./Hypermatrix_Algebra_Package_code.sage')

# Initialization of the construct whose indivdual
#  entries are boolean values in {True, False}
bA=HM([[True, False], [True,  True]]) 
bB=HM([[False, True], [True, False]])

# Computing the construct product where
# the combinator is set to : boolean disjunction
# and the composer is set to : boolean conjunction
bC=CProd([bA, bB], Or, And)
\end{sagesilent}
\[
\mathbf{bA}=\left(\begin{array}{ccc}
\text{True} &  & \text{False}\\
\\
\text{True} &  & \text{True}
\end{array}\right),\quad\mathbf{bB}=\left(\begin{array}{ccc}
\text{False} &  & \text{True}\\
\\
\text{True} &  & \text{False}
\end{array}\right),
\]
and illustrates the computation of the product
\[
\mathbf{bC}=\mbox{CProd}_{\bigvee,\,\wedge}\left(\mathbf{bA},\mathbf{bB}\right)=\left(\begin{array}{ccc}
\sage{bC[0,0]} &  & \sage{bC[0,1]}\\
\\
\sage{bC[1,0]} &  & \sage{bC[1,1]}
\end{array}\right).
\]
In particular note that for all $\mathbf{M}\in\left(\left\{ \text{True},\,\text{False}\right\} \right)^{2\times2}$,
we have 
\[
\mbox{CProd}_{\bigvee,\,\wedge}\left(\left(\begin{array}{ccc}
\text{True} &  & \text{False}\\
\\
\text{False} &  & \text{True}
\end{array}\right),\:\mathbf{M}\right)=\mathbf{M}=\mbox{GProd}_{\bigvee,\,\wedge}\left(\mathbf{M},\:\left(\begin{array}{ccc}
\text{True} &  & \text{False}\\
\\
\text{False} &  & \text{True}
\end{array}\right)\right).
\]
More generally, the identity element for CProd$_{\bigvee,\,\wedge}$
over $\left(\left\{ \text{True},\,\text{False}\right\} \right)^{n\times n}$
is the $n\times n$ construct
\[
\left(\begin{array}{cccccc}
\text{True} &  & \text{False} & \cdots &  & \text{False}\\
\\
\text{False} &  & \ddots & \ddots &  & \vdots\\
\vdots &  & \ddots & \ddots &  & \text{False}\\
\\
\text{False} &  & \cdots & \text{False} &  & \text{True}
\end{array}\right).
\]
It follows from De Morgan's laws that the identity element for CProd$_{\bigwedge,\,\vee}$
over $\left(\left\{ \text{True},\,\text{False}\right\} \right)^{n\times n}$
is the $n\times n$ construct 
\[
\left(\begin{array}{cccccc}
\text{False} &  & \text{True} & \cdots &  & \text{True}\\
\\
\text{True} &  & \ddots & \ddots &  & \vdots\\
\vdots &  & \ddots & \ddots &  & \text{True}\\
\\
\text{True} &  & \cdots & \text{True} &  & \text{False}
\end{array}\right).
\]
 A fourth illustration arises from the setting where the composer
and combinator correspond respectively to the tensor product and the
direct sum 
\[
\mathcal{F}\left(x,y\right)\,:=x\otimes y\quad\text{ and }\quad\begin{array}{c}
\\
\text{Op}\\
^{0\le\textcolor{red}{j}<\ell}
\end{array}\,:=\bigoplus_{0\le\textcolor{red}{j}<\ell}
\]
The SageMath code setup which illustrates the product constructs whose
entries are themselves matrices is as follows\\
\begin{sageverbatim}
sage: # Loading the Hypermatrix Algebra Package into SageMath
sage: load('./Hypermatrix_Algebra_Package_code.sage')
sage: 
sage: # Initialization of the symbolic variables
sage: var('a00, a10, a01, a11, b00, b10, b01, b11')
(a00, a10, a01, a11, b00, b10, b01, b11)
sage: 
sage: # Initialization of the second order constructs
sage: # whose entries are 1x1 matrices (for simplicity)
sage: A0=HM(2,2,[HM(1,1,[a00]), HM(1,1,[a10]), HM(1,1,[a01]), HM(1,1,[a11])])
sage: A0
[[[[a00]], [[a01]]], [[[a10]], [[a11]]]]
sage: 
sage: A1=HM(2,2,[HM(1,1,[b00]), HM(1,1,[b10]), HM(1,1,[b01]), HM(1,1,[b11])])
sage: A1
[[[[b00]], [[b01]]], [[[b10]], [[b11]]]]
sage: 
sage: # Computing the construct product where 
sage: # the combinator is set to : DirectSum
sage: # and the composer is set to : TensorProduct
sage: A2=CProd([A0, A1], DirectSum, TensorProduct)
sage: A2[0,0]
[[a00*b00, 0], [0, a01*b10]]
\end{sageverbatim}\\
The code above illustrates the initialization of $2\times2$ constructs
$\mathbf{A}$, $\mathbf{B}$ whose individual entries are $1\times1$
matrices 
\[
\mathbf{A}=\left(\begin{array}{ccc}
\left(a_{00}\right) &  & \left(a_{01}\right)\\
\\
\left(a_{10}\right) &  & \left(a_{11}\right)
\end{array}\right),\quad\mathbf{B}=\left(\begin{array}{ccc}
\left(b_{00}\right) &  & \left(b_{01}\right)\\
\\
\left(b_{10}\right) &  & \left(b_{11}\right)
\end{array}\right),
\]
and illustrates the computation of the product\begin{sagesilent}
# Loading the Hypermatrix Algebra Package into SageMath
load('./Hypermatrix_Algebra_Package_code.sage')

# Initialization of the symbolic variables
a00, a10, a01, a11, b00, b10, b01, b11=var('a00, a10, a01, a11, b00, b10, b01, b11')

# Initialization of the second order constructs
# whose entries are 1x1 matrices (for simplicity)
A0=HM(2,2,[HM(1,1,[a00]), HM(1,1,[a10]), HM(1,1,[a01]), HM(1,1,[a11])])
A1=HM(2,2,[HM(1,1,[b00]), HM(1,1,[b10]), HM(1,1,[b01]), HM(1,1,[b11])])

# Computing the construct product where 
# the combinator is set to : DirectSum
# and the composer is set to : TensorProduct
A2=CProd([A0, A1], DirectSum, TensorProduct)
\end{sagesilent}
\[
\mathbf{C}=\mbox{CProd}_{\bigoplus,\,\otimes}\left(\mathbf{A},\mathbf{B}\right)=\left(\begin{array}{ccc}
\sage{A2[0,0].matrix()} &  & \sage{A2[0,1].matrix()}\\
\\
\sage{A2[1,0].matrix()} &  & \sage{A2[1,1].matrix()}
\end{array}\right).
\]
Note that, replacing each matrix entry of CProd$_{\bigoplus,\,\otimes}\left(\mathbf{A},\mathbf{B}\right)$
by the corresponding trace, yields yet another way of recovering the
usual matrix product as a special instance of product of constructs.

The product of third-order constructs $\mathbf{A}$, $\mathbf{B}$
and $\mathbf{C}$ respectively of size $m\times\textcolor{red}{\ell}\times p$,
$m\times n\times\textcolor{red}{\ell}$ and $\textcolor{red}{\ell}\times n\times p$
is a construct of size $m\times n\times p$ specified entry-wise by
\[
\mbox{CProd}_{\text{Op},\mathcal{F}}\left(\mathbf{A},\mathbf{B},\mathbf{C}\right)\left[i,j,k\right]=\begin{array}{c}
\\
\text{Op}\\
^{0\le\textcolor{red}{t}<\ell}
\end{array}\,\mathcal{F}\left(\mathbf{A}\left[i,\textcolor{red}{t},k\right],\,\mathbf{B}\left[i,j,\textcolor{red}{t}\right],\,\mathbf{C}\left[\textcolor{red}{t},j,k\right]\right).
\]
The corresponding SageMath code setup is as follows\\
\begin{sageverbatim}
sage: # Loading the Hypermatrix Algebra Package into SageMath
sage: load('./Hypermatrix_Algebra_Package_code.sage')
sage: 
sage: # Initialization of the constructs which in this
sage: # particular case are symbolic 2x2x2 hypermatrices
sage: hA=HM(2,2,2,'a'); hB=HM(2,2,2,'b'); hC=HM(2,2,2,'c')
sage: hA
[[[a000, a001], [a010, a011]], [[a100, a101], [a110, a111]]]
sage: hB
[[[b000, b001], [b010, b011]], [[b100, b101], [b110, b111]]]
sage: hC
[[[c000, c001], [c010, c011]], [[c100, c101], [c110, c111]]]
sage: # Computing the construct product where
sage: # the combinator is set to : sum
sage: # and the composer is set to : prod
sage: hD=CProd([hA, hB, hC], sum, prod)
sage: # Displaying some entries on screen
sage: hD[0,0,0]
a000*b000*c000 + a010*b001*c100
\end{sageverbatim}\\
The code above illustrates the initialization of constructs 
\[
\mathbf{hA}\in\left(\mathbb{C}\left[a_{000},a_{100},\cdots,a_{011},a_{111}\right]\right)^{2\times2\times2},\ \mathbf{hB}\in\left(\mathbb{C}\left[b_{000},b_{100},\cdots,b_{011},b_{111}\right]\right)^{2\times2\times2},
\]
such that
\[
\mathbf{hA}\left[:,:,0\right]=\left(\begin{array}{rr}
a_{000} & a_{010}\\
a_{100} & a_{110}
\end{array}\right),\ \mathbf{hA}\left[:,:,1\right]=\left(\begin{array}{rr}
a_{001} & a_{011}\\
a_{101} & a_{111}
\end{array}\right),
\]
\[
\mathbf{hB}\left[:,:,0\right]=\left(\begin{array}{rr}
b_{000} & b_{010}\\
b_{100} & b_{110}
\end{array}\right),\ \mathbf{hB}\left[:,:,1\right]=\left(\begin{array}{rr}
b_{001} & b_{011}\\
b_{101} & b_{111}
\end{array}\right),
\]
and also illustrates the computation of the product

\begin{sagesilent}
# Loading the Hypermatrix Algebra Package into SageMath
load('./Hypermatrix_Algebra_Package_code.sage')

# Initialization of the constructs which in this
# particular case are symbolic 2x2x2 hypermatrices
hA=HM(2,2,2,'a'); hB=HM(2,2,2,'b'); hC=HM(2,2,2,'c')

# Computing the construct product where
# the combinator is set to : sum
# and the composer is set to : prod
hD=CProd([hA, hB, hC], sum, prod)

\end{sagesilent}
\[
\mathbf{hD}=\mbox{CProd}_{\sum,\times}\left(\mathbf{hA},\mathbf{hB},\mathbf{hC}\right)
\]
where
\[
\mathbf{hD}\left[:,:,0\right]=\left(\begin{array}{ccc}
\sage{hD[0,0,0]} &  & \sage{hD[0,1,0]}\\
\\
\sage{hD[1,0,0]} &  & \sage{hD[1,1,0]}
\end{array}\right),
\]
\[
\mathbf{hD}\left[:,:,1\right]=\left(\begin{array}{ccc}
\sage{hD[0,0,1]} &  & \sage{hD[0,1,1]}\\
\\
\sage{hD[1,0,1]} &  & \sage{hD[1,1,1]}
\end{array}\right).
\]
More generally, the product of $m$-th order \emph{constructs} generalizes
the Bhattacharya and Mesner hypermatrix product introduced in \cite{MB90,MB94}
as follows
\[
\mbox{CProd}_{\text{Op},\mathcal{F}}\left(\mathbf{A}^{(0)},\cdots,\mathbf{A}^{(t)},\cdots,\mathbf{A}^{(m-1)}\right)\left[i_{0},\cdots,i_{t},\cdots,i_{m-1}\right]=
\]
\begin{equation}
\begin{array}{c}
\\
\text{Op}\\
^{0\le\textcolor{red}{j}<\ell}
\end{array}\mathcal{F}\left(\mathbf{A}^{(0)}\left[i_{0},\textcolor{red}{j},i_{2},\cdots,i_{m-1}\right],\cdots,\mathbf{A}^{(t)}\left[i_{0},\cdots,i_{t},\textcolor{red}{j},i_{t+2},\cdots,i_{m-1}\right],\cdots,\mathbf{A}^{(m-1)}\left[\textcolor{red}{j},i_{1},\cdots,i_{m-1}\right]\right).\label{BM product-1}
\end{equation}
Note that over any field $\mathbb{K}$ equipped with a well defined
exponentiation operation, there is an external/internal composer duality
which relates special choices of composers as illustrated by the following
identities :
\[
\forall\ \mathbf{A}\in\mathbb{K}^{m\times\textcolor{red}{\ell}}\ \text{ and }\ \mathbf{B}\in\mathbb{K}^{\textcolor{red}{\ell}\times n}
\]
\[
\begin{cases}
\begin{array}{ccccc}
\text{CProd}_{\sum,\,\times}\left(\mathbf{A},\mathbf{B}\right) & = & \text{CProd}_{\sum,\,\mathcal{F}}\left(\mathbf{A}z,\,\mathbf{B}\right) & \text{ where } & \mathcal{F}\left(f\left(z\right),g\left(z\right)\right)=f\left(g\left(z\right)\right),\\
\\
\text{CProd}_{\prod,\,\exp}\left(\mathbf{A},\mathbf{B}\right) & = & \text{CProd}_{\prod,\,\mathcal{F}}\left(\mathbf{A}^{\circ^{z}},\mathbf{B}\right) & \text{ where } & \mathcal{F}\left(f\left(z\right),g\left(z\right)\right)=f\left(g\left(z\right)\right),\\
\\
\text{CProd}_{\prod,\,\text{base}\exp}\left(\mathbf{A},\mathbf{B}\right) & = & \text{CProd}_{\prod,\,\mathcal{F}}\left(z^{\circ^{\mathbf{A}}},\mathbf{B}\right) & \text{ where } & \mathcal{F}\left(f\left(z\right),g\left(z\right)\right)=f\left(g\left(z\right)\right),
\end{array}\end{cases}
\]
\[
\text{where}
\]
\[
\left(\mathbf{A}^{\circ^{z}}\right)\left[i,j\right]=\begin{cases}
\begin{array}{cc}
\left(\mathbf{A}\left[i,j\right]\right)^{z} & \text{ if }\mathbf{A}\left[i,j\right]\ne0\\
0 & \text{otherwise}
\end{array},\end{cases}\quad\left(z^{\circ^{\mathbf{A}}}\right)\left[i,j\right]=z^{\mathbf{A}\left[i,j\right]}.
\]

\section{Algebraic constructs}

\subsection{The computational model}

Determining the complexity of solutions to systems of equations is
the main motivation for algebraic constructs. The computational model
which underpins algebraic constructs is based upon combinational and
arithmetic circuits \cite{AZ,Burgisser,DG,GRS,GZ,Koiran,Svaiter}.
Recall that arithmetic circuits are finite directed acyclic graphs
in which non-input nodes (or gates) are labeled with arithmetic operations
$\left\{ +,\,\times,\,\mathcircumflex,\,\frac{\partial}{\partial x},\,\text{log},\,\text{mod}\right\} $
respectively associated with addition, multiplication, exponentiation,
partial derivative, logarithm and the modular gates. For simplicity
each gate has fan-in equal to two. The partial derivative gate returns
the partial derivative of the right input where the left input specfies
the order of the derivative. The logarithm gate outputs the multivalued
logarithmic orbit associated with the logarithm of the right input
where the left input specifies the base of the logarithm. The modular
gates output the multivalued remainder orbit associated with the Euclidean
division of the left input by the right input. Note that the exponentiation
gate also outputs a multivalued orbit. Labels of input node are taken
from the set $\left\{ -1,1\right\} $ as well as a finite set of symbolic
variables. Variables in this model are seen as place holders for arithmetical
circuits whose determination might or should be postponed. Recall
that arithmetic formulas are special arithmetic circuits. The directed
acyclic graph of formulas are rooted trees whose edges are oriented
to point towards the root. Because, formulas are a simpler class of
circuits, our discussion will focus mainly on formulas. Given $f,g,h$
arbitrary valid arithmetic formulas from the model, equivalence classes
which relate valid arithmetic formulas are prescribed by transformation
rules which include
\begin{itemize}
\item Commutativity 
\[
\begin{array}{c}
f+g\longleftrightarrow g+f\\
f\times g\longleftrightarrow g\times f
\end{array},
\]
\item Associativity 
\[
\begin{array}{c}
\left(f+g\right)+h\longleftrightarrow f+\left(g+h\right)\\
\left(f\times g\right)\times h\longleftrightarrow f\times\left(g\times h\right)
\end{array},
\]
\item Distributivity
\[
\begin{array}{c}
f\times\left(g+h\right)\longleftrightarrow f\times g+f\times h\\
f\mathcircumflex\left(g+h\right)\longleftrightarrow f\mathcircumflex g\times f\mathcircumflex h\\
\left(f\times g\right)\mathcircumflex h\longleftrightarrow f\mathcircumflex h\times g\mathcircumflex h
\end{array},
\]
\item Unit element 
\[
\begin{array}{ccc}
f\times1 & \longleftrightarrow & f\\
f\mathcircumflex1 & \longleftrightarrow & f\\
1\mathcircumflex f & \longleftrightarrow & 1\\
f+\left(-1+1\right) & \longleftrightarrow & f\\
f\times\left(-1+1\right) & \longleftrightarrow & -1+1\\
f\mathcircumflex\left(-1+1\right) & \longleftrightarrow & 1
\end{array},
\]
\item Log gate 
\[
\begin{array}{ccc}
\text{log}\left(h^{f}\times h^{g},h\right) & \longleftrightarrow & f+g\\
\text{log}\left(f\mathcircumflex g,f\right) & \longleftrightarrow & g\\
\text{log}\left(1,f\right) & \longleftrightarrow & -1+1
\end{array},
\]
\item Mod gate 
\[
\begin{array}{ccc}
\text{mod}\left(\text{mod}\left(f,h\right)+\text{mod}\left(g,h\right),h\right) & \longleftrightarrow & \text{mod}\left(f+g,h\right)\\
\text{mod}\left(\text{mod}\left(f,h\right)\times\text{mod}\left(g,h\right),h\right) & \longleftrightarrow & \text{mod}\left(f\times g,h\right)\\
\text{mod}\left(-1+1,f\right) & \longleftrightarrow & -1+1
\end{array},
\]
\item Partial derivative gate
\[
\begin{array}{ccc}
\frac{\partial}{\partial x_{i}}\left(f+g\right) & \longleftrightarrow & \frac{\partial}{\partial x_{i}}f+\frac{\partial}{\partial x_{i}}g\\
\frac{\partial}{\partial x_{i}}\left(f\times g\right) & \longleftrightarrow & g\times\frac{\partial}{\partial x_{i}}f+f\times\frac{\partial}{\partial x_{i}}g
\end{array}.
\]
\end{itemize}
In the partial listing of transformation rules specified above, the
notation $f\longleftrightarrow g$ posits that the arithmetic formulas
$f$ and $g$ belong to the same equivalence class. Similarly the
notation $f\not\leftrightarrow g$ posits that the arithmetic formulas
$f$ and $g$ belong to distinct equivalence classes. The model invalidates
any circuits which admits a sub-circuit which belongs to forbidden
equivalence classes which include the class of formulas associated
$\left(-1+1\right)\mathcircumflex\left(-1\right)$ or $\left(-1+1\right)\mathcircumflex\left(-1+1\right)$
or log$\left(-1+1,f\right)$ or log$\left(f,-1+1\right)$ or mod$\left(f,\,-1+1\right)$. 

\subsection{Arithmetic formula complexity}

Every valid formula, admits a canonical \textbf{prefix} string encoding.
In particular if the formula of interest is free of any symbolic variable,
then the corresponding prefix encoding is a string made up of characters
from the alphabet $\left\{ -1,\,1,\,+,\,\times,\,\mathcircumflex,\,\text{log},\,\text{mod}\right\} $.
For example 
\[
\text{the prefix string encoding of the formula}\:\text{log}_{1+1}\left\{ (1+1)^{(1+1)}\right\} \text{ is the string }\text{log}+11\mathcircumflex+11+11.
\]
The size of a given formula refers to the length of its prefix string
encoding. The complexity of a given formula refers to the minimal
size of any formula which lies in the same equivalence class. For
illustration purposes, we describe a set recurrence formula. The proposed
set recurrence generalizes combinatorial constructions introduced
in \cite{GZ,GRS} for determining the complexity of formula encoding
of numbers. More precisely, the recurrence stratifies formulas according
to their complexity. The initial conditions for the set recurrence
are specified by
\[
A_{0}=\emptyset,\,A_{1}=\left\{ -1,1\right\} ,\,A_{2}=\emptyset.
\]
The recurrence relation prescribes for all integers $n>2$, a finite
sets of arithmetic formulas. For simplicity, we further restrict gates
appearing in the formula to $\left\{ +,\,\times,\,\mathcircumflex\right\} $.
Let $A_{n}^{+}$, $A_{n}^{\times}$, $A_{n}^{\mathcircumflex}$ denote
sets of formulas whose root node correspond to an addition ($+$),
a multiplication ($\times$) and an exponentiation ($\mathcircumflex$)
gate respectively :
\[
A_{n}^{+}=\bigcup_{\begin{array}{c}
s\in A_{i},\,t\in A_{n-i-1}\\
s+t\not\leftrightarrow-1+1
\end{array}}\left(s+t\right),
\]
\[
A_{n}^{\times}=\bigcup_{\begin{array}{c}
s\in A_{i},\,t\in A_{n-i-1}\\
s\times t\not\leftrightarrow1
\end{array}}\left(s\times t\right),
\]
\[
A_{n}^{\mathcircumflex}=\bigcup_{\begin{array}{c}
s\in A_{i},\,t\in A_{n-i-1}\\
s^{t}\not\leftrightarrow1
\end{array}}\left(s^{t}\right),
\]
\begin{equation}
A_{n}=A_{n}^{+}\cup A_{n}^{\times}\cup A_{n}^{\mathcircumflex}.\label{Set recurrence}
\end{equation}
By construction, the set $A_{n}$ collects all reduced arithmetic
formulas whose gates are restricted to $\left\{ +,\,\times,\,\mathcircumflex\right\} $
( which are free of any symbolic variable ) of size $n$. A formula
free of any variables is \emph{reduced} if none of its sub-formulas
of size $>2$ lie in the equivalence classes of either $1$ or $\left(-1+1\right)$.
Recall that a formula free of any variable is said to be \emph{monotone}
if the formula does not make use of the input $-1$. The set recurrence
prescribed above, determines the non-monotone complexity of relatively
small formulas. For instance the set recurrence above establishes
that in this restricted setting the non-monotone complexity of the
integer $-512$ is equal to $13$ which is also the complexity of
of any solution to the algebraic equation
\[
10\,x^{6}+98\,x^{5}+379\,x^{4}+712\,x^{3}+671\,x^{2}+290\,x+50=0.
\]
Note that for all integers $n>2$, the following bounds on the cardinality
of $A_{n}$ noted $\left|A_{n}\right|$ holds
\[
\begin{cases}
\begin{array}{cccccc}
 &  & \left|A_{n}\right| & = & 0 & \text{ if }n\equiv0\mod2\\
\frac{2}{n+1}{n-1 \choose \frac{n-1}{2}} & < & \left|A_{n}\right| & < & 5^{n} & \text{otherwise}
\end{array} & .\end{cases}
\]
While some system of algebraic equation admit solution expressible
as monotone formula encodings it of interest to quantify the reduction
in size achieved by obtaining minimal non-monotone encodings of such
solutions. The following conjecture is motivated by this question
\begin{conjecture}
The asymptotic gap between the monotone and non-monotone complexity
of the positive integer $a_{n}$ in the sequence 
\[
a_{0}=2-1,\quad a_{n+1}=2^{\left(1+a_{n}\right)}-1
\]
of formula encoding whose gates and inputs are respectively restricted
to $\left\{ +,\,\times,\,\mathcircumflex\,\right\} $ and $\left\{ -1,\,1\right\} $
is $\widetilde{O}\left(a_{n-1}\right)$ allowing for poly-logarithmic
factors.
\end{conjecture}

Note that the size of the shortest non-monotone formula of $a_{n}$
is upper-bounded by $4\left(n+1\right)+1$ for $n>0$. The exponential
growth relative to $n$ of $\left|A_{n}\right|$ calls for alternative
approaches to bounding the complexity of solutions to systems of equations.
The SageMath code setup for the set recurrence relation described
above is as follows\\
\begin{sageverbatim}
sage: # Loading the Hypermatrix Algebra Package into SageMath
sage: load('./Hypermatrix_Algebra_Package_code.sage')
sage: 
sage: # Computing the recurrence for numbers
sage: # whose non monotone formula
sage: # encoding have size at most 9
sage: L=ReducedNonMonotoneFormulaSets(9)
sage: L[:8]
[{},
 {1, -1},
 {},
 {2, -2},
 {},
 {3, -1/2, -3, 1/2},
 {},
 {-3/2, 4, I, 3/2, -1/3, 1/3, 1/4, -I, -4}]
\end{sageverbatim}\\
The code above yields
\[
A_{0}=A_{2}=A_{4}=A_{6}=A_{8}=\emptyset
\]
\[
A_{1}=\left\{ 1,-1\right\} ,\,A_{3}=\left\{ 2,-2\right\} ,\,A_{5}=\left\{ 3,-\frac{1}{2},-3,\frac{1}{2}\right\} ,\,A_{7}=\left\{ -\frac{3}{2},4,\sqrt{-1},\frac{3}{2},-\frac{1}{3},\frac{1}{3},\frac{1}{4},-\sqrt{-1},-4\right\} ,
\]
\[
A_{9}=\left\{ -\frac{5}{2},\sqrt{-2},5,-\frac{4}{3},-\frac{2}{3},\left(-1\right)^{\frac{1}{3}},\frac{5}{2},\frac{2}{3},\frac{4}{3},-\frac{1}{4},\frac{5}{4},-\frac{1}{2}\,\sqrt{-2},-\frac{1}{8},\frac{1}{8},6,\frac{1}{9},\sqrt{2},\left(-1\right)^{-\sqrt{-1}},\right.
\]
\[
\left.,-\sqrt{-1}-1,-\sqrt{-1}+1,\left(-1\right)^{\sqrt{-1}},\sqrt{-1}-1,\sqrt{-1}+1,-\frac{3}{4},\left(-1\right)^{\frac{1}{4}},\frac{1}{2}\,\sqrt{2},9,-\left(-1\right)^{\frac{2}{3}},8,-8,-6,-5\right\} .
\]
As an aside, we describe a recursive construction for determining
the roots of polynomial equations with rational coefficients. Recall
that the algebraically complete set of multilinear symmetric polynomials
are the densest set of symmetric polynomials having in total $2^{n}$
distinct monomial terms. While the algebraically complete set of Newton
elementary symmetric polynomials are the sparsest set of symmetric
polynomial in the roots having in total $n^{2}$ distinct monomial
terms. As such they are convenient starting point for the construction
of formal recursive expression of solutions to algebraic equations
Let the composer be given by
\[
\left\{ \sum_{0\le i<n}\left(x_{i}\right)^{j}=\sum_{0\le i<n}\left(r_{i}\right)^{j}\::\:j\in\left[1,n\right]\cap\mathbb{Z}\right\} ,
\]
equivalently expressed as 
\[
\left\{ \text{Cprod}_{\sum,\times}\left(\mathbf{1}_{1\times n}z,\:\text{Cprod}_{\prod,\text{BaseExp}}\left(j\mathbf{I}_{n},\,\mathbf{x}\right)\right)=\sigma_{j}\left(\mathbf{r}\right)\::\:j\in\left[1,n\right]\cap\mathbb{Z}\right\} ,
\]
where $\sigma_{j}\left(\mathbf{r}\right)$ denotes the $j$-th Newton
symmetric polynomial in the roots. The constraints are thus summarized
as follows 
\begin{equation}
\text{Cprod}_{\sum,\times}\left(\left(\mathbf{I}_{n}\otimes\mathbf{1}_{1\times n}\right),\:\text{Cprod}_{\prod,\text{BaseExp}}\left(\left(\begin{array}{c}
1\\
2\\
\vdots\\
n
\end{array}\right)\otimes\mathbf{I}_{n},\,\mathbf{x}\right)\right)-\boldsymbol{\sigma}\left(\mathbf{r}\right)=\mathbf{0}_{n\times1}.\label{Newton Symmetric Equation}
\end{equation}
For convenience we adopt the notation convention 
\[
\text{Cprod}_{\prod,\text{BaseExp}}\left(\mathbf{A},\,\mathbf{x}\right)=\mathbf{x}^{\mathbf{A}}
\]
where $\mathbf{x}$ is $n\times1$ and $\mathbf{A}$ is $m\times n$.
The constraints in Eq. (\ref{Newton Symmetric Equation}) are therefore
more simply rewritten as 
\begin{equation}
\left(\mathbf{I}_{n}\otimes\mathbf{1}_{1\times n}\right)\,\mathbf{x}^{\left(\left(\begin{array}{c}
1\\
2\\
\vdots\\
n
\end{array}\right)\otimes\mathbf{I}_{n}\right)}-\boldsymbol{\sigma}\left(\mathbf{r}\right)=\mathbf{0}_{n\times1}
\end{equation}
It is well known that solutions to univariate polynomial can be expressed
as hypergeometric functions of the coefficients \cite{sturmfelssolving,STURMFELS2000171}.
We describe a refinement of Newton's iterative method. Our propose
method expresses the roots as an infinite composition sequence. The
SageMath code setup for the expression above is as follows\\
\begin{sageverbatim}
sage: # Loading the Hypermatrix Algebra Package into SageMath
sage: load('./Hypermatrix_Algebra_Package_code.sage')
sage: 
sage: # Initialization of the size parameter
sage: sz=3
sage: # Initialization of the vector of unknowns
sage: # and of the symbolic variables associated with roots.
sage: X=HM(sz,1,var_list('x',sz)); vZ=HM(sz,1,var_list('z',sz))
sage: X.printHM()
[:, :]=
[x0]
[x1]
[x2]
sage: vZ.printHM()
[:, :]=
[z0]
[z1]
[z2]
sage: # Initialization of the identity matrix
sage: Id=HM(2,sz,'kronecker')
sage: # Initialization of the exponent matrix
sage: A=HM(sz,1,rg(1,sz+1)).tensor_product(Id)
sage: A.printHM()
[:, :]=
[1 0 0]
[0 1 0]
[0 0 1]
[2 0 0]
[0 2 0]
[0 0 2]
[3 0 0]
[0 3 0]
[0 0 3]
sage: # Initialization of the coeficient matrix
sage: B=Id.tensor_product(HM(1,sz,'one'))
sage: B.printHM()
[:, :]=
[1 1 1 0 0 0 0 0 0]
[0 0 0 1 1 1 0 0 0]
[0 0 0 0 0 0 1 1 1]
sage: (B*X^A-B*vZ^A).printHM()
[:, :]=
[            x0 + x1 + x2 - z0 - z1 - z2]
[x0^2 + x1^2 + x2^2 - z0^2 - z1^2 - z2^2]
[x0^3 + x1^3 + x2^3 - z0^3 - z1^3 - z2^3]
\end{sageverbatim}\\
For instance in the case $n=3$, Eq. (\ref{Newton Symmetric Equation})
is of the form 
\[
\left(\begin{array}{rrrrrrrrr}
1 & 1 & 1 & 0 & 0 & 0 & 0 & 0 & 0\\
0 & 0 & 0 & 1 & 1 & 1 & 0 & 0 & 0\\
0 & 0 & 0 & 0 & 0 & 0 & 1 & 1 & 1
\end{array}\right)\cdot\left(\begin{array}{r}
x_{0}\\
x_{1}\\
x_{2}
\end{array}\right)^{\left(\begin{array}{rrr}
1 & 0 & 0\\
0 & 1 & 0\\
0 & 0 & 1\\
2 & 0 & 0\\
0 & 2 & 0\\
0 & 0 & 2\\
3 & 0 & 0\\
0 & 3 & 0\\
0 & 0 & 3
\end{array}\right)}=\left(\begin{array}{rrrrrrrrr}
1 & 1 & 1 & 0 & 0 & 0 & 0 & 0 & 0\\
0 & 0 & 0 & 1 & 1 & 1 & 0 & 0 & 0\\
0 & 0 & 0 & 0 & 0 & 0 & 1 & 1 & 1
\end{array}\right)\cdot\left(\begin{array}{r}
r_{0}\\
r_{1}\\
r_{2}
\end{array}\right)^{\left(\begin{array}{rrr}
1 & 0 & 0\\
0 & 1 & 0\\
0 & 0 & 1\\
2 & 0 & 0\\
0 & 2 & 0\\
0 & 0 & 2\\
3 & 0 & 0\\
0 & 3 & 0\\
0 & 0 & 3
\end{array}\right)}
\]
 
\[
\sum_{0<k\le n}\mathbf{A}_{n-k}\cdot\left(\mathbf{x}-\mathbf{a}\right)^{k\,\mathbf{I}_{n}}=\boldsymbol{\sigma}\left(\mathbf{r}\right)-\boldsymbol{\sigma}\left(\mathbf{a}\right)
\]
where 
\begin{equation}
\mathbf{A}_{n-k}\left[j,:\right]={j \choose k}\:\left(\mathbf{a}^{\left(j-k\right)\mathbf{I}_{n}}\right)^{\top}.\label{coefficient matrices}
\end{equation}
In expressing the rows of each matrix $\mathbf{A}_{n-k}$ in Eq. (\ref{coefficient matrices}),
we adopt the notation convention 
\[
\forall\:k>j,\quad{j \choose k}=0,
\]
it follows that
\[
\text{Rank}\left(\mathbf{A}_{k}\right)=k+1\;\text{ and }\det\left(\mathbf{A}_{n-1}\right)=n!\prod_{0\le i<j<n}\left(a_{j}-a_{i}\right).
\]
The infinite composition sequence which determine the roots is expressed
by the equality
\[
\mathbf{x}-\mathbf{a}=\mathbf{A}_{n-1}^{-1}\left(\left(\boldsymbol{\sigma}\left(\mathbf{x}\right)-\boldsymbol{\sigma}\left(\mathbf{a}\right)\right)-\sum_{1<j\le d}\mathbf{A}_{n-k}\left(\mathbf{x}-\mathbf{a}\right)^{k\,\mathbf{I}_{n}}\right)
\]
The composition sequence is thus obtained by successively substituting
every occurrence of $\left(\mathbf{x}-\mathbf{a}\right)$ on the right-hand
side by the entire expression obtained at the previous iteration.
Consider the induced vector mapping 
\[
F_{\mathbf{x},\mathbf{a}}\,:\,\mathbb{C}^{n\times1}\rightarrow\mathbb{C}^{n\times1}
\]
\[
\text{such that }
\]
\[
F_{\mathbf{x},\mathbf{a}}\left(\mathbf{z}\right)=\mathbf{A}_{n-1}^{-1}\left(\left(\boldsymbol{\sigma}\left(\mathbf{x}\right)-\boldsymbol{\sigma}\left(\mathbf{a}\right)\right)-\sum_{1<j\le d}\mathbf{A}_{n-k}\,\mathbf{z}^{k\,\mathbf{I}_{n}}\right),
\]
then the expansion of a solution around the neighborhood of $\mathbf{a}$
is expressed by 
\[
\mathbf{x}=\mathbf{a}+\lim_{t\rightarrow\infty}F_{\mathbf{x},\mathbf{a}}^{\left(t\right)}\left(\mathbf{z}\right).
\]
 In the case where the roots of the corresponding univariate polynomials
are all distinct, the vector $\mathbf{a}$ breaks the symmetry among
the $n!$ possible solutions as suggested by the following conjecture.
\begin{conjecture}
The iteration expresses the particular permutation of the roots on
the complex plane if each entry of $\mathbf{a}$ lies in a distinct
Voronoi cell associated with the root induced Voronoi partition of
the complex plane for some metric.
\end{conjecture}

\subsection{Combinational formula complexity}

We now emphasize compelling similarities which relate arithmetic formulas
to boolean formulas. Combinational circuits are simpler than arithmetic
circuit. The comparative simplicity of combinational circuits results
from the fact that boolean circuits always express uni-valued functions
and also because every boolean functional can be encoded as a circuit
made up only of four gates. The AND gate noted ($\wedge$), the OR
gate noted ($\vee$) each having fan-in equal to two and the NOT gate
noted ($\neg$) having fan-in one. The inputs to boolean formulas
are restricted to boolean values 
\[
\left\{ \text{True},\ \text{False}\right\} .
\]
By contrast to arithmetic circuits, boolean formulas free from any
variables are trivial. Similarly to arithmetic formula, Every boolean
formula on $n$ input variables is uniquely encoded by a prefix string
encoding made up of characters from the alphabet 
\[
\left\{ \vee,\,\wedge,\,\neg,\,\text{True},\,\text{False},\,x_{0},\,x_{1},\,\cdots,x_{n-1}\right\} .
\]
By analogy to arithmetic formulas, we describe set recurrence relation
which stratifies boolean formulas according to their complexity. The
set recurrence is prescribed by initial conditions
\[
A_{0}=\emptyset,\,A_{1}=\left\{ x_{0}\right\} ,\,A_{2}=\emptyset.
\]
Let $A_{n}^{\vee}$, $A_{n}^{\wedge}$, $A_{n}^{\neg}$ respectively
denote sets of formulas whose root node is respectively an OR ($\vee$),
an AND ($\wedge$) and a NOT ($\neg$) gate. The set recurrence formulas
is
\[
A_{n}^{\vee}=\bigcup_{\begin{array}{c}
s\in A_{i},\,t\in A_{n-i-1}\\
s\vee\text{Increment Var Index}\left(t,j\right)\text{ is new}
\end{array}}\left(\bigcup_{0\le j\le\#\text{vars in }s}\left(s\vee\text{Increment Var Index}\left(t,j\right)\right)\right),
\]
\[
A_{n}^{\wedge}=\bigcup_{\begin{array}{c}
s\in A_{i},\,t\in A_{n-i-1}\\
s\wedge\text{Increment Var Index}\left(t,j\right)\text{ is new}
\end{array}}\left(\bigcup_{0\le j\le\#\text{vars in }s}\left(s\wedge\text{Increment Var Index}\left(t,j\right)\right)\right),
\]
\[
A_{n}^{\neg}=\bigcup_{\begin{array}{c}
s\in A_{n-1}\\
\neg s\text{ is new}
\end{array}}\left(\neg s\right),
\]
\[
A_{n}=A_{n}^{\vee}\cup A_{n}^{\wedge}\cup A_{n}^{\neg}.
\]
In the recurrence formula above the operation Increment Var Index,
increments the index of every variable in the boolean function specified
as the left input by the non-negative integer specified as the right
input. The SageMath code setup for the set recurrence relation described
above is as follows\\
\begin{sageverbatim}
sage: # Loading the Hypermatrix Algebra Package into SageMath
sage: load('./Hypermatrix_Algebra_Package_code.sage')
sage: 
sage: # Computing the recurrence for numbers
sage: # which number whose non monotone formula
sage: # encoding have size at most 5
sage: # A stores the formulas and L store their 
sage: # lexicographic numbering 
sage: [A, L]=ReducedNonMonotoneBooleanFormula(5)
sage: A
[[],
 [x0],
 [['NOT', x0]],
 [['AND', x0, x1], ['OR', x0, x1]],
 [['AND', x0, ['NOT', x0]],
  ['AND', x0, ['NOT', x1]],
  ['AND', ['NOT', x0], x1],
  ['OR', x0, ['NOT', x0]],
  ['OR', x0, ['NOT', x1]],
  ['OR', ['NOT', x0], x1],
  ['NOT', ['AND', x0, x1]],
  ['NOT', ['OR', x0, x1]]],
 [['AND', x0, ['AND', x1, x2]],
  ['AND', x0, ['OR', x0, x1]],
  ['AND', x0, ['OR', x1, x2]],
  ['AND', ['OR', x0, x1], x1],
  ['AND', ['OR', x0, x1], x2],
  ['OR', x0, ['AND', x1, x2]],
  ['OR', x0, ['OR', x1, x2]],
  ['OR', ['AND', x0, x1], x2]]]
sage: L
[2, 1, 12, 18, 0, 6, 8, 3, 15, 17, 11, 5, 148, 14, 188, 16, 244, 254, 274, 268]
\end{sageverbatim}\\
The code above yields
\[
A_{0}=\emptyset
\]
\[
A_{1}=\left\{ x_{0}\right\} ,\:A_{2}=\left\{ \left[\text{\texttt{NOT}},x_{0}\right]\right\} ,\ A_{3}=\left\{ \left[\text{\texttt{AND}},x_{0},x_{1}\right],\,\left[\text{\texttt{OR}},x_{0},x_{1}\right]\right\} 
\]
\[
A_{4}=\left\{ \left[\text{\texttt{AND}},x_{0},\left[\text{\texttt{NOT}},x_{0}\right]\right],\,\left[\text{\texttt{AND}},x_{0},\left[\text{\texttt{NOT}},x_{1}\right]\right],\,\left[\text{\texttt{AND}},\left[\text{\texttt{NOT}},x_{0}\right],x_{1}\right],\left[\text{\texttt{OR}},x_{0},\left[\text{\texttt{NOT}},x_{0}\right]\right],\right.
\]
\[
\left.\left[\text{\texttt{OR}},x_{0},\left[\text{\texttt{NOT}},x_{1}\right]\right],\,\left[\text{\texttt{OR}},\left[\text{\texttt{NOT}},x_{0}\right],x_{1}\right],\,\left[\text{\texttt{NOT}},\left[\text{\texttt{AND}},x_{0},x_{1}\right]\right],\,\left[\text{\texttt{NOT}},\left[\text{\texttt{OR}},x_{0},x_{1}\right]\right]\right\} ,
\]
\[
A_{5}=\left\{ \left[\text{\texttt{AND}},x_{0},\left[\text{\texttt{AND}},x_{1},x_{2}\right]\right],\,\left[\text{\texttt{AND}},x_{0},\left[\text{\texttt{OR}},x_{0},x_{1}\right]\right],\,\left[\text{\texttt{AND}},x_{0},\left[\text{\texttt{OR}},x_{1},x_{2}\right]\right],\,\left[\text{\texttt{AND}},\left[\text{\texttt{OR}},x_{0},x_{1}\right],x_{1}\right],\right.
\]
\[
\left.\left[\text{\texttt{AND}},x_{0},\left[\text{\texttt{AND}},x_{1},x_{2}\right]\right],\,\left[\text{\texttt{AND}},x_{0},\left[\text{\texttt{OR}},x_{0},x_{1}\right]\right],\,\left[\text{\texttt{AND}},x_{0},\left[\text{\texttt{OR}},x_{1},x_{2}\right]\right],\,\left[\text{\texttt{AND}},\left[\text{\texttt{OR}},x_{0},x_{1}\right],x_{1}\right]\right\} .
\]
The square brackets are meant to more clearly delineate sub-formulas.
Note that every boolean formula in $n$ variables can be converted
into an arithmetic formula which encodes a multivariate polynomial
in $\mathbb{Z}\left[x_{0},\cdots,x_{n-1}\right]$ via the following
correspondence
\[
\begin{cases}
\begin{array}{ccc}
\text{True} & \leftrightarrow & 1\\
\\
\text{False} & \leftrightarrow & 0\\
\\
\neg\,x_{i} & \leftrightarrow & 1-x_{i}\\
\\
x_{i}\vee x_{j} & \leftrightarrow & x_{i}+x_{j}-x_{i}\cdot x_{j}\\
\\
x_{i}\wedge x_{j} & \leftrightarrow & x_{i}\cdot x_{j}
\end{array}.\end{cases}
\]
For instance the algebraic expression associated with the combinational
circuits obtained in the previous SageMath experimental setup is obtained
by the following code setup\\
\begin{sageverbatim}
sage: # Loading the Hypermatrix Algebra Package into SageMath
sage: load('./Hypermatrix_Algebra_Package_code.sage')
sage: 
sage: # Computing the recurrence for numbers
sage: # which number whose non monotone formula
sage: # encoding have size at most 5
sage: # A stores the formulas and L store their 
sage: # lexicographic numbering 
sage: [A, L]=ReducedNonMonotoneBooleanFormulaPoly(5)
sage: A
[[],
 [x0],
 [-x0 + 1],
 [x0*x1, -x0*x1 + x0 + x1],
 [-(x0 - 1)*x0,
  -x0*(x1 - 1),
  -(x0 - 1)*x1,
  (x0 - 1)*x0 + 1,
  x0*(x1 - 1) + x0 - x1 + 1,
  (x0 - 1)*x1 - x0 + x1 + 1,
  -x0*x1 + 1,
  x0*x1 - x0 - x1 + 1]]
sage: L
[2, 1, 12, 18, 0, 6, 8, 3, 15, 17, 11, 5, 148, 14, 188, 16, 244, 254, 274, 268]
\end{sageverbatim}\\
Every equivalence class of elements $\mathbb{Z}\left[x_{0},\cdots,x_{n-1}\right]$
associated with a given boolean function in the entries of $\mathbf{x}$
has a unique canonical multilinear representative obtained by reducing
the polynomial modulo the algebraic relations $\left\{ x_{i}^{2}\equiv x_{i}\,:\,i\in\left[0,n\right)\cap\mathbb{Z}\right\} $.
An injective mapping from boolean functions to non-negative integers
is called a lexicographic ordering of boolean function. A natural
lexicographic ordering of boolean functions is such that $\forall\:F\,:\,\left\{ 0,1\right\} ^{n\times1}\rightarrow\left\{ 0,1\right\} $
we have 
\[
\text{lex}\left(0\right)=0,\ \text{lex}\left(1\right)=1,
\]
and more generally 
\[
\text{lex}\left(F\left(\mathbf{x}\right)\right)=\left(\sum_{0<i<n}2^{2^{i}}\right)+\sum_{\mathbf{b}\in\left\{ 0,1\right\} ^{n\times1}}F\left(\mathbf{b}\right)\prod_{0\le j<n}2^{b_{j}2^{j}},
\]
\[
\implies\sum_{0<i<n}2^{2^{i}}\le\text{lex}\left(F\left(\mathbf{x}\right)\right)<\sum_{0<i\le n}2^{2^{i}}.
\]
The SageMath code setup for obtaining the lexicographic numbering
of boolean formula is as follows\\
\begin{sageverbatim}
sage: # Loading the Hypermatrix Algebra Package into SageMath
sage: load('./Hypermatrix_Algebra_Package_code.sage')
sage: 
sage: # Computing the lexicographic ordering of
sage: # the input boolean formula
sage: Bool2Integer(['AND',['NOT',x0],x0])
0
sage: Bool2Integer(['OR',['NOT',x0],x0])
3
sage: Bool2Integer(x0)
2
sage: Bool2Integer(['NOT',x0])
1
\end{sageverbatim} 

\subsection{Systems of equations associated with CProd$_{\sum}$}

For the rest of section 2, unless otherwise specified, we set the
\emph{composer }of second order constructs to 
\[
\mathcal{F}\,:\,\mathbb{K}^{\mathbb{K}}\times\mathbb{K}^{\mathbb{K}}\rightarrow\mathbb{K}^{\mathbb{K}}
\]
 ( for some commutative or skew field $\mathbb{K}$ ) such that
\begin{equation}
\mathcal{F}\left(f\left(z\right),g\left(z\right)\right)=f\left(g\left(z\right)\right),\quad\forall\,f,\,g\,\in\,\mathbb{K}^{\mathbb{K}}.\label{Composer}
\end{equation}

Three types of systems of equations form the basis for algebraic constructs.
The first type corresponds to systems of linear equations obtained
by setting the combinator to
\[
\begin{array}{c}
\\
\text{Op}\\
^{0\le\textcolor{red}{j}<k}
\end{array}\,:=\sum_{0\le\textcolor{red}{j}<k}.
\]
Consequently, the product of conformable constructs
\[
\mathbf{A}\left(z\right)\in\left(\mathbb{C}^{\mathbb{C}}\right)^{n_{0}\times\ell}\quad\text{ and }\quad\mathbf{B}\left(z\right)\in\left(\mathbb{C}^{\mathbb{C}}\right)^{\ell\times n_{1}},
\]
whose entries are given by 
\[
\begin{array}{cc}
\mathbf{A}\left(z\right)\left[i_{0},t\right]= & a_{i_{0}\,t}\,z+0\,z^{0}\\
\\
\mathbf{B}\left(z\right)\left[t,i_{1}\right]= & 0\,z+b_{t\,i_{1}}\,z^{0}
\end{array},\;\forall\:\begin{cases}
\begin{array}{cc}
0\le i_{0}<n_{0}\\
0\le\,t<\,\ell\\
0\le i_{1}<n_{1}
\end{array}\end{cases}
\]
recovers the usual matrix product as
\[
\mbox{CProd}_{\sum}\left(\mathbf{A}\left(z\right),\mathbf{B}\left(z\right)\right)=\mathbf{A}\left(1\right)\cdot\mathbf{B}\left(z\right).
\]
For instance, let 
\[
\mathbf{A}\left(z\right)=\left(\begin{array}{ccc}
a_{00}\,z+0\,z^{0} &  & a_{01}\,z+0\,z^{0}\\
\\
a_{10}\,z+0\,z^{0} &  & a_{11}\,z+0\,z^{0}
\end{array}\right)\ \text{ and }\ \mathbf{B}\left(z\right)=\left(\begin{array}{ccc}
0\,z+b_{00}\,z^{0} &  & 0\,z+b_{01}\,z^{0}\\
\\
0\,z+b_{10}\,z^{0} &  & 0\,z+b_{11}\,z^{0}
\end{array}\right),
\]
whose trivial terms $0\,z$ and $0\,z^{0}$ are meant to emphasize
the fact that every entry ( including the constant entries ) are to
be viewed as polynomials in the morphism variable $z$. In particular
\[
\mbox{CProd}_{\sum}\left(\mathbf{A}\left(z\right),\mathbf{B}\left(z\right)\right)=\left(\begin{array}{rcr}
a_{00}b_{00}+a_{01}b_{10} &  & a_{00}b_{01}+a_{01}b_{11}\\
\\
a_{10}b_{00}+a_{11}b_{10} &  & a_{10}b_{01}+a_{11}b_{11}
\end{array}\right)=\mathbf{A}\left(1\right)\cdot\mathbf{B}\left(z\right).
\]
As a result, for all $\mathbf{A}\left(z\right)\in\left(\mathbb{C}^{\mathbb{C}}\right)^{n\times n}$,
the identity construct for CProd$_{\Sigma}$ is the construct $z\mathbf{I}_{n}$
prescribed by
\[
\mbox{CProd}_{\sum}\left(\mathbf{A}\left(z\right)-\mathbf{A}\left(0\right),\mathbf{I}\left(z\right)\right)=\mathbf{A}\left(z\right)-\mathbf{A}\left(0\right)=\mbox{CProd}_{\sum}\left(\mathbf{I}\left(z\right),\mathbf{A}\left(z\right)-\mathbf{A}\left(0\right)\right),
\]
 For example when $n=2$ the identity construct is 
\[
\mathbf{I}\left(z\right)=z\mathbf{I}_{2}=\left(\begin{array}{cc}
z+0\,z^{0} & 0\,z+0z^{0}\\
0\,z+0z^{0} & z+0z^{0}
\end{array}\right).
\]
Let $\mathbf{A}\left(z\right)\in\left(\mathbb{C}^{\mathbb{C}}\right)^{2\times2}$
be given by 
\[
\mathbf{A}\left(z\right)=\left(\begin{array}{ccc}
a_{00}\,z+0\,z^{0} &  & a_{01}\,z+0\,z^{0}\\
\\
a_{10}\,z+0\,z^{0} &  & a_{11}\,z+0\,z^{0}
\end{array}\right),
\]
then
\[
\mbox{CProd}_{\sum}\left(\mathbf{A}\left(z\right),\mathbf{I}\left(z\right)\right)=\mathbf{A}\left(z\right)=\mbox{CProd}_{\sum}\left(\mathbf{I}\left(z\right),\mathbf{A}\left(z\right)\right).
\]
The inverse of a second order construct is illustrated by the equalities
\[
\mbox{CProd}_{\sum}\left\{ \left(\begin{array}{rcr}
a_{00}\,z+0z^{0} &  & a_{01}\,z+0\,z^{0}\\
\\
a_{10}\,z+0z^{0} &  & a_{11}\,z+0\,z^{0}
\end{array}\right),\,\left(\begin{array}{rcr}
\frac{-a_{11}z+0z^{0}}{a_{01}a_{10}-a_{00}a_{11}} &  & \frac{a_{01}z+0z^{0}}{a_{01}a_{10}-a_{00}a_{11}}\\
\\
\frac{a_{10}z+0z^{0}}{a_{01}a_{10}-a_{00}a_{11}} &  & \frac{-a_{00}z+0z^{0}}{a_{01}a_{10}-a_{00}a_{11}}
\end{array}\right)\right\} =\mathbf{I}\left(z\right),
\]
\[
\text{and}
\]
\[
\mbox{CProd}_{\sum}\left\{ \left(\begin{array}{rcr}
\frac{-a_{11}z+0z^{0}}{a_{01}a_{10}-a_{00}a_{11}} &  & \frac{a_{01}z+0z^{0}}{a_{01}a_{10}-a_{00}a_{11}}\\
\\
\frac{a_{10}z+0z^{0}}{a_{01}a_{10}-a_{00}a_{11}} &  & \frac{-a_{00}z+0z^{0}}{a_{01}a_{10}-a_{00}a_{11}}
\end{array}\right),\,\left(\begin{array}{rcr}
a_{00}z+0z^{0} &  & a_{01}z+0z^{0}\\
\\
a_{10}z+0z^{0} &  & a_{11}z+0z^{0}
\end{array}\right)\right\} =\mathbf{I}\left(z\right).
\]

CProd$_{\Sigma}$ canonically expresses a system of linear equations
as 
\[
\mathbf{0}_{m\times1}=\mbox{CProd}_{\sum}\left(\mathbf{A}\left(z\right),\mathbf{x}\left(z\right)\right),
\]
where
\[
\mathbf{A}\left(z\right)\left[i,j\right]=a_{ij}\,z-\frac{b_{i}}{n}\ \text{ and }\ \mathbf{x}\left(z\right)\left[j\right]=0\,z+x_{j}\,z^{0}\quad\forall\,\begin{cases}
\begin{array}{c}
0\le i<m\\
0\le j<n
\end{array}\end{cases}.
\]
$\mathbf{A}\left(z\right)\in\left(\mathbb{C}^{\mathbb{C}}\right)^{m\times n}$
and $\mathbf{x}\left(z\right)\in\left(\mathbb{C}^{\mathbb{C}}\right)^{n\times1}$
respectively denote coefficient and variable constructs of the system.
For instance a system of two equations in the unknowns $x_{0}$, $x_{1}$
and $x_{2}$ is expressed in terms of coefficient and variable constructs
\[
\mathbf{A}\left(z\right)=\left(\begin{array}{rcrcr}
a_{00}z-\frac{b_{0}}{3} &  & a_{01}z-\frac{b_{0}}{3} &  & a_{02}z-\frac{b_{0}}{3}\\
\\
a_{10}z-\frac{b_{1}}{3} &  & a_{11}z-\frac{b_{1}}{3} &  & a_{12}z-\frac{b_{1}}{3}
\end{array}\right),\quad\mathbf{x}\left(z\right)=\left(\begin{array}{c}
0\,z+x_{0}z^{0}\\
\\
0\,z+x_{1}z^{0}\\
\\
0\,z+x_{2}z^{0}
\end{array}\right),
\]
\[
\text{as}
\]
\[
\mathbf{0}_{2\times1}=\mbox{CProd}_{\sum}\left(\mathbf{A}\left(z\right),\mathbf{x}\left(z\right)\right)=\left(\begin{array}{r}
0\,z+\left(a_{00}x_{0}+a_{01}x_{1}+a_{02}x_{2}-b_{0}\right)z^{0}\\
\\
0\,z+\left(a_{10}x_{0}+a_{11}x_{1}+a_{12}x_{2}-b_{1}\right)z^{0}
\end{array}\right).
\]
The corresponding SageMath code setup is as follows\\
\begin{sageverbatim}

sage: # Loading the Hypermatrix Algebra Package into SageMath
sage: load('./Hypermatrix_Algebra_Package_code.sage')
sage: 
sage: # Initialization of the morphism variable
sage: z=var('z')
sage: 
sage: # Initialization of the order and size parameters
sage: od=2; sz=2
sage: 
sage: # Initialization of the constructs
sage: M0=z*HM(sz, sz, 'a') + HM(sz, sz, 'b')
sage: M0
[[a00*z + b00, a01*z + b01], [a10*z + b10, a11*z + b11]]
sage: M1=z*HM(sz, sz, 'c') + HM(sz, sz, 'd')
sage: M1
[[c00*z + d00, c01*z + d01], [c10*z + d10, c11*z + d11]]
sage: 
sage: # Computing the products
sage: # The product of construct where the composer is
sage: # set by default to the composition of functions
sage: # is implemented as GProd instead of CProd
sage: # and it is used as follows.
sage: M2=GProd([M0, M1], sum, [z]).expand()
sage: M3=GProd([M1, M0], sum, [z]).expand()
sage: 
sage: # Initializing the right identity construct
sage: rId=z*HM(od,sz,'kronecker') - i2x2(HM(sz,sz,'a'))*HM([[b01,b00],[b11,b10]])
sage: 
sage: # Initializing the left identity construct
sage: lId=z*HM(od, sz, 'kronecker')
sage: lId
[[z, 0], [0, z]]
sage: # Computing the product with the right identity
sage: M4=GProd([M0, rId], sum, [z]).factor()
sage: M4
[[a00*z + b00, a01*z + b01], [a10*z + b10, a11*z + b11]]
sage: # Computing the product with the left identity
sage: M5=GProd([lId, M1], sum, [z])
sage: M5
[[c00*z + d00, c01*z + d01], [c10*z + d10, c11*z + d11]]
\end{sageverbatim}\\
The code setup above initializes the constructs\begin{sagesilent}

# Loading the Hypermatrix Algebra Package into SageMath
load('./Hypermatrix_Algebra_Package_code.sage')

# Initialization of the morphism variable
z=var('z')

# Initialization of the order and size parameters
od=2; sz=2

# Initialization of the constructs
M0=z*HM(sz, sz, 'a') + HM(sz, sz, 'b')
M1=z*HM(sz, sz, 'c') + HM(sz, sz, 'd')

# Computing the products
M2=GProd([M0, M1], sum, [z]).expand()
M3=GProd([M1, M0], sum, [z]).expand()

# Initializing the right identity construct
rId=z*HM(od,sz,'kronecker') - i2x2(HM(sz,sz,'a'))*HM([[b01,b00],[b11,b10]])

# Initializing the left identity construct
lId=z*HM(od, sz, 'kronecker')

# Computing the product with the right identity
M4=GProd([M0, rId], sum, [z]).factor()

# Computing the product with the left identity
M5=GProd([lId, M1], sum, [z])

\end{sagesilent}
\[
\mathbf{M}_{0}\left(z\right)=\sage{M0.matrix()},\:\mathbf{M}_{1}\left(z\right)=\sage{M1.matrix()},
\]
\[
\mathbf{lId}\left(z\right)=\sage{lId.matrix()},
\]
\[
\text{and}
\]
\[
\mathbf{rId}\left(z\right)=\sage{rId.matrix()}.
\]
The code setup also computes the products
\[
\text{GProd}_{\sum}\left(\mathbf{M}_{0}\left(z\right),\mathbf{M}_{1}\left(z\right)\right)=
\]
\[
\sage{M2.matrix()},
\]
\[
\text{GProd}_{\sum}\left(\mathbf{M}_{1}\left(z\right),\mathbf{M}_{0}\left(z\right)\right)=
\]
\[
\sage{M3.matrix()}.
\]
Finally the code setup illustrates subtle differences between left
and right identity elements via products
\[
\text{CProd}_{\sum}\left(\mathbf{M}_{0}\left(z\right),\mathbf{rId}\left(z\right)\right)=\sage{M4.matrix()},
\]
\[
\text{CProd}_{\sum}\left(\mathbf{lId}\left(z\right),\mathbf{M}_{1}\left(z\right)\right)=\sage{M5.matrix()}.
\]
These examples establish that even in the simple setting of constructs
from $\left(\mathbb{C}^{\mathbb{C}}\right)^{2\times2}$ whose entries
are polynomials of degree at most one in the morphism variable $z$,
the left identity elements may differ from the right identity elements.
Moreover, right identity elements may differ for distinct elements
of $\left(\mathbb{C}^{\mathbb{C}}\right)^{2\times2}$. Methods for
solving system of linear equations are easily adapted to the construct
formulation.\\
As an illustration let $\mathbf{A}\left(z\right)\in\left(\mathbb{K}^{\mathbb{K}}\right)^{2\times2}$
denote a coefficient construct and $\mathbf{x}\left(z\right)\in\left(\mathbb{K}^{\mathbb{K}}\right)^{2\times1}$
a variable construct defined over some skew field $\mathbb{K}$ such
that
\begin{equation}
\mathbf{A}\left(z\right)=\left(\begin{array}{cc}
a_{00}z+\left(\frac{-c_{0}}{2}\right) & a_{01}z+\left(\frac{-c_{0}}{2}\right)\\
a_{10}z+\left(\frac{-c_{1}}{2}\right) & a_{11}z+\left(\frac{-c_{1}}{2}\right)
\end{array}\right),\ \mathbf{x}\left(z\right)=\left(\begin{array}{c}
x_{0}\\
x_{1}
\end{array}\right).
\end{equation}
The corresponding linear system is 
\begin{equation}
\mathbf{0}_{2\times1}=\mbox{CProd}_{\sum}\left(\mathbf{A}\left(z\right),\mathbf{x}\left(z\right)\right)\Leftrightarrow\begin{cases}
\begin{array}{cccccc}
0= & a_{00}x_{0} & + & a_{01}x_{1} & + & \left(-1\right)c_{0}\\
0= & a_{10}x_{0} & + & a_{11}x_{1} & + & \left(-1\right)c_{1}
\end{array}\end{cases}.\label{Example of left system}
\end{equation}
Using Gaussian elimination, solutions to such a system are expressed
via four row operations. The first three row operations put the system
in Row Echelon Form (REF) and the last row operation puts the system
in Reduced Row Echelon Form (RREF). \\
The first row operation is the row linear combination 
\begin{equation}
-a_{10}a_{00}^{-1}\text{R}_{0}+\text{R}_{1}\rightarrow\text{R}_{1}\label{Row linear combination}
\end{equation}
which yields
\begin{equation}
\begin{cases}
\begin{array}{cccccc}
0= & a_{00}x_{0} & + & a_{01}x_{1} & + & \left(-1\right)c_{0}\\
0= & \left(-a_{10}a_{00}^{-1}a_{00}+a_{10}\right)x_{0} & + & \left(-a_{10}a_{00}^{-1}a_{01}+a_{11}\right)x_{1} & + & \left(-1\right)\left(-a_{10}a_{00}^{-1}c_{0}+c_{1}\right)
\end{array}\end{cases}
\end{equation}
\begin{equation}
\implies\begin{cases}
\begin{array}{cccccc}
0= & a_{00}x_{0} & + & a_{01}x_{1} & + & \left(-1\right)c_{0}\\
0= & 0x_{0} & + & \left(-a_{10}a_{00}^{-1}a_{01}+a_{11}\right)x_{1} & + & \left(-1\right)\left(-a_{10}a_{00}^{-1}c_{0}+c_{1}\right)
\end{array}\end{cases}.\label{system prior to row scaling}
\end{equation}
The next two row operations are row scaling operations
\begin{equation}
\begin{array}{ccc}
a_{00}^{-1}\text{R}_{0} & \rightarrow & \text{R}_{0}\\
\left(-a_{10}a_{00}^{-1}a_{01}+a_{11}\right)^{-1}\text{R}_{1} & \rightarrow & \text{R}_{1}
\end{array}\label{Row scaling}
\end{equation}
which yield
\begin{equation}
\begin{cases}
\begin{array}{ccc}
0= & a_{00}^{-1}a_{00}x_{0}+a_{00}^{-1}a_{01}x_{1}+\left(-1\right)a_{00}^{-1}c_{0}\\
0= & 0x_{0}+\left(-a_{10}a_{00}^{-1}a_{01}+a_{11}\right)^{-1}\left(-a_{10}a_{00}^{-1}a_{01}+a_{11}\right)x_{1} & +\left(-1\right)\left(-a_{10}a_{00}^{-1}a_{01}+a_{11}\right)^{-1}\left(-a_{10}a_{00}^{-1}c_{0}+c_{1}\right)
\end{array} & .\end{cases}
\end{equation}
\begin{equation}
\implies\begin{cases}
\begin{array}{cccccc}
0= & x_{0} & + & a_{00}^{-1}a_{01}x_{1} & + & \left(-1\right)a_{00}^{-1}c_{0}\\
0= & 0x_{0} & + & x_{1} & + & \left(-1\right)\left(-a_{10}a_{00}^{-1}a_{01}+a_{11}\right)^{-1}\left(-a_{10}a_{00}^{-1}c_{0}+c_{1}\right)
\end{array} & .\end{cases}
\end{equation}
The row operations described in the illustration above which put the
system in REF can be performed directly on the constructs $\mathbf{A}\left(z\right)$
and $\mathbf{x}\left(z\right)$. More specifically the row linear
combination operation (\ref{Row linear combination}) changes $\mathbf{A}\left(z\right)$
to
\begin{equation}
\mathbf{A}_{0}\left(z\right)=\left(\begin{array}{ccc}
a_{00}z+\left(\frac{-c_{0}}{2}\right) &  & a_{01}z+\left(\frac{-c_{0}}{2}\right)\\
\\
0\,z+\left(\frac{-\left(-a_{10}a_{00}^{-1}c_{0}+c_{1}\right)}{2}\right) &  & \left(-a_{10}a_{00}^{-1}a_{01}+a_{11}\right)z+\left(\frac{-\left(-a_{10}a_{00}^{-1}c_{0}+c_{1}\right)}{2}\right)
\end{array}\right).
\end{equation}
On the other hand, the effect of row scaling operations are more easily
achieved by performing instead an invertible change of variable. For
instance, the effect of the row scaling operations (\ref{Row scaling})
are more easily obtained by performing the following change of variables
in (\ref{system prior to row scaling})
\begin{equation}
\begin{cases}
\begin{array}{ccc}
x_{0} & = & a_{00}^{-1}y_{0}\\
x_{1} & = & \left(-a_{10}a_{00}^{-1}a_{01}+a_{11}\right)^{-1}y_{1}
\end{array} & .\end{cases}
\end{equation}
\begin{equation}
\implies\mathbf{x}\left(z\right)=\left(\begin{array}{c}
a_{00}^{-1}\left(a_{00}x_{0}\right)\\
\left(-a_{10}a_{00}^{-1}a_{01}+a_{11}\right)^{-1}\left(\left(-a_{10}a_{00}^{-1}a_{01}+a_{11}\right)x_{1}\right)
\end{array}\right)=\left(\begin{array}{c}
a_{00}^{-1}y_{0}\\
\left(-a_{10}a_{00}^{-1}a_{01}+a_{11}\right)^{-1}y_{1}
\end{array}\right)=\mathbf{y}^{*}\left(z\right).
\end{equation}
The resulting system in the new variables $y_{0}$ and $y_{1}$ is
in REF and given by 
\begin{equation}
\mathbf{0}_{2\times1}=\mbox{CProd}_{\sum}\left(\mathbf{A}_{0}\left(z\right),\mathbf{y}^{\prime}\left(z\right)\right)
\end{equation}
which yields 
\begin{equation}
\begin{cases}
\begin{array}{ccccccc}
0 & = & y_{0} & + & a_{01}\left(-a_{10}a_{00}^{-1}a_{01}+a_{11}\right)^{-1}y_{1} & + & \left(-1\right)c_{0}\\
0 & = & 0y_{0} & + & y_{1} & + & \left(-1\right)\left(-a_{10}a_{00}^{-1}c_{0}+a_{00}c_{1}\right)
\end{array}\end{cases}.
\end{equation}
Finally, the RREF of the system in the new variables $y_{0}$ and
$y_{1}$ is obtained by performing the row linear combination operation
\[
-a_{01}\left(-a_{10}a_{00}^{-1}a_{01}+a_{11}\right)^{-1}\text{R}_{1}+\text{R}_{0}\rightarrow\text{R}_{0}
\]
which yields 
\begin{equation}
\begin{cases}
\begin{array}{ccccccc}
0 & = & y_{0} & + & 0y_{1} & + & \left(-1\right)\left(-a_{01}\left(-a_{10}a_{00}^{-1}a_{01}+a_{11}\right)^{-1}\left(-a_{10}a_{00}^{-1}c_{0}+a_{00}c_{1}\right)+c_{0}\right)\\
0 & = & 0y_{0} & + & y_{1} & + & \left(-1\right)\left(-a_{10}a_{00}^{-1}c_{0}+a_{00}c_{1}\right)
\end{array} & .\end{cases}
\end{equation}
Once all solutions in the new variables have been obtained, one derives
from such solutions the solutions in the original variables by successively
inverting the changes of variables. An important benefit of the construct
formulation of systems of linear equations over skew fields is the
fact it de-emphasize the distinctions between left system ( systems
where the coefficients multiply all variables on the left ), right
system ( systems where the coefficients multiply all the variables
on the right ) and mixed systems ( systems where the coefficients
may multiply variables both on the left and right ). For instance
the system in Eq. (\ref{Example of left system}) is a left system.
The right system analog of Eq. (\ref{Example of left system}) is
associated with coefficient and variable constructs $\mathbf{A}\left(z\right)\in\left(\mathbb{K}^{\mathbb{K}}\right)^{2\times2}$
and $\mathbf{x}\left(z\right)\in\left(\mathbb{K}^{\mathbb{K}}\right)^{2\times1}$
respectively such that 
\begin{equation}
\mathbf{A}\left(z\right)=\left(\begin{array}{cc}
zb_{00}+\left(\frac{-c_{0}}{2}\right) & zb_{10}+\left(\frac{-c_{0}}{2}\right)\\
zb_{01}+\left(\frac{-c_{1}}{2}\right) & zb_{11}+\left(\frac{-c_{1}}{2}\right)
\end{array}\right),\ \mathbf{x}\left(z\right)=\left(\begin{array}{c}
x_{0}\\
x_{1}
\end{array}\right).
\end{equation}
 The corresponding system is 
\begin{equation}
\mathbf{0}_{2\times1}=\mbox{CProd}_{\sum}\left(\mathbf{A}\left(z\right),\mathbf{x}\left(z\right)\right)\Leftrightarrow\begin{cases}
\begin{array}{ccccccc}
0 & = & x_{0}b_{00} & + & x_{1}b_{10} & + & \left(-1\right)c_{0}\\
0 & = & x_{0}b_{01} & + & x_{1}b_{11} & + & \left(-1\right)c_{1}
\end{array}\end{cases}.
\end{equation}
The associated constructs is put in REF via the row linear combination
operation
\begin{equation}
\mathbf{A}\left(z\right)\begin{array}{c}
\longrightarrow\\
_{_{-\text{R}_{0}b_{00}^{-1}b_{01}+\text{R}_{1}\rightarrow\text{R}_{1}}}
\end{array}\mathbf{A}_{0}\left(z\right)=\left(\begin{array}{ccc}
zb_{00}+\left(\frac{-c_{0}}{2}\right) &  & zb_{10}+\left(\frac{-c_{0}}{2}\right)\\
\\
z\,0+\left(\frac{-\left(c_{0}b_{00}^{-1}b_{01}+c_{1}\right)}{2}\right) &  & z\left(-b_{10}b_{00}^{-1}b_{01}+b_{11}\right)+\left(\frac{-\left(c_{0}b_{00}^{-1}b_{01}+c_{1}\right)}{2}\right)
\end{array}\right)
\end{equation}
followed by the change of variable which effect the scaling
\begin{equation}
\begin{cases}
\begin{array}{ccc}
x_{0} & = & y_{0}b_{00}^{-1}\\
x_{1} & = & y_{1}\left(-b_{10}b_{00}^{-1}b_{01}+b_{11}\right)^{-1}
\end{array} & ,\end{cases}
\end{equation}
\begin{equation}
\implies\mathbf{x}\left(z\right)=\left(\begin{array}{c}
\left(x_{0}b_{00}\right)b_{00}^{-1}\\
\left(x_{1}\left(-b_{10}b_{00}^{-1}b_{01}+b_{11}\right)\right)\left(-b_{10}b_{00}^{-1}b_{01}+b_{11}\right)^{-1}
\end{array}\right)=\left(\begin{array}{c}
y_{0}b_{00}^{-1}\\
y_{1}\left(-b_{10}b_{00}^{-1}b_{01}+b_{11}\right)^{-1}
\end{array}\right)=\mathbf{y}^{\prime}\left(z\right).
\end{equation}
Finally, the system is put in RREF via the row linear combination
operation 
\begin{equation}
\mathbf{A}_{0}\left(z\right)\begin{array}{c}
\longrightarrow\\
_{_{_{-\text{R}_{1}\left(\left(-b_{10}b_{00}^{-1}b_{01}+b_{11}\right)^{-1}b_{10}\right)+\text{R}_{0}\rightarrow\text{R}_{0}}}}
\end{array}\mathbf{A}_{1}\left(z\right)=\left(\begin{array}{ccc}
z+\left(\frac{-d_{0}}{2}\right) &  & z0+\left(\frac{-d_{0}}{2}\right)\\
\\
z\,0+\left(\frac{-d_{1}}{2}\right) &  & z+\left(\frac{-d_{1}}{2}\right)
\end{array}\right)
\end{equation}
where 
\begin{equation}
d_{0}=c_{0}-\left(c_{0}b_{00}^{-1}b_{01}+c_{1}\right)\left(\left(-b_{10}b_{00}^{-1}b_{01}+b_{11}\right)^{-1}b_{10}\right),\quad d_{1}=c_{0}b_{00}^{-1}b_{01}+c_{1}
\end{equation}
Note that both left and right systems are special cases of the more
general mixed linear systems of equation. It is well known that solutions
to left systems as well as to right systems are expressible as rational
functions of the coefficients of $z$ in $\mathbf{A}\left(z\right)$
\cite{GR} when they exist. More generally, a mixed linear system
of two equations in the two unknowns $x_{0}$ and $x_{1}$ is specified
with the coefficient and variable constructs respectively of the form
\begin{equation}
\mathbf{A}\left(z\right)=\left(\begin{array}{ccc}
a_{00}\,z\,b_{00}+\left(\frac{-c_{0}}{2}\right) &  & a_{01}\,z\,b_{10}+\left(\frac{-c_{0}}{2}\right)\\
\\
a_{10}\,z\,b_{01}+\left(\frac{-c_{1}}{2}\right) &  & a_{11}\,z\,b_{11}+\left(\frac{-c_{1}}{2}\right)
\end{array}\right)\:\text{ and }\:\mathbf{x}\left(z\right)=\left(\begin{array}{c}
x_{0}\\
x_{1}
\end{array}\right).
\end{equation}
The corresponding system is
\begin{equation}
\mathbf{0}_{m\times1}=\mbox{CProd}_{\sum}\left(\mathbf{A}\left(z\right),\mathbf{x}\left(z\right)\right).\label{System of equations}
\end{equation}
More explicitly expressed as 
\begin{equation}
\begin{cases}
\begin{array}{cccccc}
0= & a_{00}\,x_{0}\,b_{00} & + & a_{01}\,x_{1}\,b_{10} & + & \left(-1\right)c_{0}\\
0= & a_{10}\,x_{0}\,b_{01} & + & a_{11}\,x_{1}\,b_{11} & + & \left(-1\right)c_{1}
\end{array}\end{cases}.
\end{equation}
The following code snippets illustrates how to set up a mixed system
using SageMath\\
\begin{sageverbatim}
sage: # Loading the Hypermatrix package
sage: load('./Hypermatrix_Algebra_tst.sage')
sage: 
sage: # Initialization of the size parameter and the variables
sage: sz=2; l=2
sage: 
sage: # Initialization of the morphism variable
sage: z=var('z')
sage: 
sage: # Defining the list of variables
sage: La=HM(sz,l,'a').list(); Lb=HM(sz,l,'b').list()
sage: Lc=var_list('c',sz); Lx=var_list('x',l)
sage: 
sage: # Initialization of the Free algebra
sage: F=FreeAlgebra(QQ,len(La+Lx+Lb+Lc+[z]),La+Lx+Lb+Lc+[z])
sage: F.<a00,a10,a01,a11,x0,x1,b00,b10,b01,b11,c0,c1,z>=\
....: FreeAlgebra(QQ,len(La+Lx+Lb+Lc+[z]))
sage: 
sage: # Initialization of some temporary matrices used to initialize the constructs
sage: Ha=HM(sz,l,[a00, a10, a01, a11])
sage: Hb=HM(sz,l,[b00, b10, b01, b11]).transpose()
sage: 
sage: # Initialization of the constructs
sage: A=Ha.elementwise_product(z*Hb)-HM(sz,1,[c0,c1])*HM(1,l,[QQ(1/2) for i in rg(l)])
sage: A.printHM()
[:, :]=
[-1/2*c0 + a00*z*b00 -1/2*c0 + a01*z*b10]
[-1/2*c1 + a10*z*b01 -1/2*c1 + a11*z*b11]
sage: X=HM(l,1,[x0,x1])
sage: X.printHM()
[:, :]=
[x0]
[x1]
sage: # Computing the product
sage: M=GeneralHypermatrixProductIV([A, X], sum, [z])
sage: M.printHM()
[:, :]=
[-c0 + a00*x0*b00 + a01*x1*b10]
[-c1 + a10*x0*b01 + a11*x1*b11]
\end{sageverbatim}\\
Running the code above initializes the constructs\begin{sagesilent}

# Loading the Hypermatrix package
load('./Hypermatrix_Algebra_tst.sage')

# Initialization of the size parameter and the variables
sz=2; l=2

# Initialization of the morphism variable
z=var('z')

# Defining the list of variables
La=HM(sz,l,'a').list(); Lb=HM(sz,l,'b').list()
Lc=var_list('c',sz); Lx=var_list('x',l)

# Initialization of the Free algebra
F=FreeAlgebra(QQ,len(La+Lx+Lb+Lc+[z]),La+Lx+Lb+Lc+[z])
F.<a00,a10,a01,a11,x0,x1,b00,b10,b01,b11,c0,c1,z>=FreeAlgebra(QQ,len(La+Lx+Lb+Lc+[z]))

# Initialization of some temporary matrices used to initialize the constructs
Ha=HM(sz,l,[a00, a10, a01, a11])
Hb=HM(sz,l,[b00, b10, b01, b11]).transpose()

# Initialization of the constructs
A=Ha.elementwise_product(z*Hb)-HM(sz,1,[c0,c1])*HM(1,l,[QQ(1/2) for i in rg(l)])
X=HM(l,1,[x0,x1])

# Computing the product
M=GeneralHypermatrixProductIV([A, X], sum, [z])

\end{sagesilent}
\begin{equation}
\mathbf{A}\left(z\right)=\sage{A.matrix()},\quad\mathbf{x}\left(z\right)=\sage{X.matrix()}
\end{equation}
and the constraint
\begin{equation}
\mathbf{0}_{2\times1}=\mbox{CProd}_{\sum}\left(\mathbf{A}\left(z\right),\mathbf{x}\left(z\right)\right)
\end{equation}
as
\begin{equation}
\left(\begin{array}{c}
0\\
0
\end{array}\right)=\sage{M.matrix()}.
\end{equation}
In particular a sufficient condition for the expressibility of a unique
solution as rational functions of the coefficients of $z$ in $\mathbf{A}\left(z\right)$
is as follows.
\begin{prop}
Let $\mathbf{A}\left(z\right)\in\left(\mathbb{K}^{\mathbb{K}}\right)^{m\times n}$
and $\mathbf{x}\left(z\right)\in\left(\mathbb{K}^{\mathbb{K}}\right)^{n\times1}$
respectively denote the coefficient and variable constructs given
by
\[
\mathbf{A}\left(z\right)\left[i,j\right]=a_{ij}\,z\,b_{ji}+\left(\frac{-c_{i}}{n}\right),\;\mathbf{x}\left(z\right)\left[j\right]=x_{j},\,\forall\:\begin{cases}
\begin{array}{c}
0\le i<m\\
0\le j<n
\end{array}\end{cases}
\]
The entries of a solution to
\[
\mathbf{0}_{m\times1}=\mbox{CProd}_{\sum}\left(\mathbf{A}\left(z\right),\mathbf{x}\left(z\right)\right),
\]
are expressible as as rational functions of the coefficients $\mathbf{A}\left(z\right)$
if it results from an arbitrary but finite composition of left and
right invertible constructs.
\end{prop}

\begin{proof}
The proof follows by Gaussian elimination.
\end{proof}
The following theorem addresses the solvability of such general mixed
linear systems over an arbitrary skew field $\mathbb{K}$.
\begin{thm}
Let Let $\mathbf{A}\left(z\right)\in\left(\mathbb{K}^{\mathbb{K}}\right)^{m\times n}$
and $\mathbf{x}\left(z\right)\in\left(\mathbb{K}^{\mathbb{K}}\right)^{n\times1}$
respectively denote the coefficient and variable constructs given
by
\begin{equation}
\mathbf{A}\left(z\right)\left[i,j\right]=a_{ij}\,z\,b_{ji}+\left(\frac{-c_{i}}{n}\right),\;\mathbf{x}\left(z\right)\left[j\right]=x_{j},\,\forall\:\begin{cases}
\begin{array}{c}
0\le i<m\\
0\le j<n
\end{array}\end{cases}\label{General linear system}
\end{equation}
The entries of the general solution to
\[
\mathbf{0}_{m\times1}=\mbox{CProd}_{\sum}\left(\mathbf{A}\left(z\right),\mathbf{x}\left(z\right)\right),
\]
cannot be expressed as rational functions of the coefficients of $z$
in the entries of $\mathbf{A}\left(z\right)$.
\end{thm}

\begin{proof}
Our proof will be set over the skew field of generic symbolic $m\times m$
( left skew field ) and $n\times n$ ( right skew field ) matrices.
It suffices to prove the claim for a specific family of mixed systems.
Note that Sylvester's equation can be expressed in term of coefficient
and variable constructs respectively given by 
\begin{equation}
\mathbf{A}\left(z\right)=\left(\begin{array}{ccc}
a\,z\,\text{id}_{n}+\left(-\frac{c}{2}\right) &  & \text{id}_{m}\,z\,b+\left(-\frac{c}{2}\right)\\
\\
\text{id}_{m}\,z\,\text{id}_{n}+\left(-\frac{0}{2}\right) &  & \left(-\text{id}_{m}\right)\,z\,\text{id}_{n}+\left(-\frac{0}{2}\right)
\end{array}\right)\:\text{ and }\:\mathbf{x}\left(z\right)=\left(\begin{array}{c}
x_{0}\\
x_{1}
\end{array}\right).
\end{equation}
where id$_{m}$ denotes the identity element of the left skew field
and id$_{n}$ denotes identity element of the right skew field. It
is well known that the system admits a unique solution iff 
\begin{equation}
\text{Resultant}_{x}\left\{ \det\left(x\,\text{id}_{m}-a\right),\,\det\left(x\,\text{id}_{n}+b\right)\right\} \ne0
\end{equation}
moreover exploiting the underlying matrix structure we know that 
\begin{equation}
m_{n\cdot s+t}\left[n\cdot i+j,\,n\cdot u+v\right]=\begin{cases}
\begin{array}{cc}
\left(\text{id}_{m}\otimes a+b^{\top}\otimes\text{id}_{n}\right)\left[n\cdot i+j,\,n\cdot u+v\right] & \text{ if }n\cdot u+v\ne n\cdot s+t\\
c\left[i,j\right] & \text{otherwise}
\end{array}\end{cases}
\end{equation}
and in particular the solution is expressed by 
\begin{equation}
x\left[s,t\right]=\frac{\det\left(m_{n\cdot s+t}\right)}{\det\left(\text{id}_{m}\otimes a+b^{\top}\otimes\text{id}_{n}\right)}.
\end{equation}
It is easy to see that such expressions are not expressible as non-commutative
rational functions of the coefficients $a$ and $b$, thereby concluding
our proof.
\end{proof}
We now proceed to illustrate starting from a mixed linear system in
Eq. (\ref{System of equations}) the steps used to determine formal
expression of solutions over an arbitrary skew field ( \ref{General linear system}
). The first row operation is 
\begin{equation}
\mathbf{A}\left(z\right)\begin{array}{c}
\longrightarrow\\
_{_{_{-a_{10}a_{00}^{-1}\text{R}_{0}b_{00}^{-1}b_{01}+\text{R}_{1}\rightarrow\text{R}_{1}}}}
\end{array}\mathbf{A}_{0}\left(z\right)
\end{equation}
yields
\begin{equation}
\mathbf{A}_{0}\left(z\right)=\left(\begin{array}{ccc}
\left(\mathbf{I}_{2}\otimes a_{00}\right)\,z\,\left(\mathbf{I}_{2}\otimes b_{00}\right)+\left(\frac{-\mathbf{I}_{2}\otimes c_{0}}{2}\right) &  & \left(\mathbf{I}_{2}\otimes a_{01}\right)\,z\,\left(\mathbf{I}_{2}\otimes b_{10}\right)+\left(\frac{-\mathbf{I}_{2}\otimes c_{0}}{2}\right)\\
\\
0_{2\times2}\,z\,0_{2\times2}-\frac{\left(f-a_{10}a_{00}^{-1}c_{0}b_{00}^{-1}b_{01}\right)\oplus\left(c_{1}-f\right)}{2} &  & \left(a_{10}a_{00}^{-1}\oplus a_{11}\right)\,z\,\left(b_{00}^{-1}b_{01}\oplus b_{11}\right)-\frac{\left(f-a_{10}a_{00}^{-1}c_{0}b_{00}^{-1}b_{01}\right)\oplus\left(c_{1}-f\right)}{2}
\end{array}\right)
\end{equation}
and 
\begin{equation}
\mathbf{x}\left(z\right)\rightarrow\left(\begin{array}{c}
\mathbf{I}_{2}\otimes x_{0}\\
\mathbf{I}_{2}\otimes x_{1}
\end{array}\right).
\end{equation}
\begin{equation}
\implies\mathbf{I}_{2}\otimes x_{1}=\left(a_{10}a_{00}^{-1}\oplus a_{11}\right)^{-1}\left(\frac{\left(f-a_{10}a_{00}^{-1}c_{0}b_{00}^{-1}b_{01}\right)\oplus\left(c_{1}-f\right)}{2}\right)\left(b_{00}^{-1}b_{01}\oplus b_{11}\right)^{-1}
\end{equation}
\begin{equation}
\implies f=a_{10}a_{00}^{-1}c_{0}b_{00}^{-1}b_{01}+a_{10}a_{00}^{-1}a_{11}^{-1}\left(c_{1}-f\right)b_{11}^{-1}b_{00}^{-1}b_{01}
\end{equation}
The equality expresses a rational series determined by the recurrence
\begin{equation}
f_{0}=f,\quad f_{k+1}=a_{10}a_{00}^{-1}c_{0}b_{00}^{-1}b_{01}+a_{10}a_{00}^{-1}a_{11}^{-1}\left(c_{1}-f_{k}\right)b_{11}^{-1}b_{00}^{-1}b_{01}
\end{equation}
As ansatz we assert that the free variable is a rational series expansion
of the coefficients of $\mathbf{A}\left(z\right)$ determined by the
limit
\[
\lim_{k\rightarrow\infty}f_{k}.
\]
This illustrates a non commutative version of the Lagrange inversion
formula. Having thus expressed $x_{1}$ in terms of the coefficients
of $\mathbf{A}\left(z\right)$ back substitution yields $x_{0}$.\\
\\
In summary, a system of equations of the first type is expressed in
terms of a coefficient construct $\mathbf{A}\left(z\right)\in\left(\mathbb{K}^{\mathbb{K}}\right)^{m\times n}$
and a variable construct $\mathbf{x}\left(z\right)$ of size $n\times1$
is given by 
\[
\mathbf{A}\left(z\right)\left[i,j\right]=a_{ij}\,z\,b_{ji}+\left(\frac{-c_{i}}{n}\right),\;\mathbf{x}\left(z\right)\left[j\right]=x_{j},\;\forall\:\begin{array}{c}
0\le i<m\\
0\le j<n
\end{array}.
\]
The corresponding system is 
\[
\mathbf{0}_{m\times1}=\mbox{CProd}_{\sum}\left(\mathbf{A}\left(z\right),\mathbf{x}\left(z\right)\right).
\]
The existence of solutions expressible as formal series of the coefficients
of $z$ in the construct $\mathbf{A}\left(z\right)$ is established
by combining the following three fundamental of row operations. The
first fundamental row operations are row linear combinations, specified
as follows 
\begin{equation}
\alpha\,\text{R}_{i}\,\beta+\text{R}_{j}\rightarrow\text{R}_{j},
\end{equation}
for some non-zero $\alpha,\beta\in\mathbb{K}$. The second fundamental
of row operations are row exchanges specified for $i\ne j$ by
\begin{equation}
\text{R}_{i}\leftrightarrow\text{R}_{j}.
\end{equation}
 Note that row linear combination and row exchanges affect only the
coefficient construct $\mathbf{A}\left(z\right)$. The third fundamental
row operations are row scaling which are achieved via variable change
in $\mathbf{x}\left(z\right)$ expressed by 
\begin{equation}
x_{i}=\alpha^{-1}\left(\alpha\,x_{i}\,\beta\right)\beta^{-1}=\alpha^{-1}y_{i}\beta^{-1}=y_{i}^{*}.
\end{equation}

The second type of system of equations corresponds to log-linear system
of equations expressed in terms of exponent and variable construct
respectively denoted $\mathbf{A}\left(z\right)\in\left(\mathbb{C}^{\mathbb{C}}\right)^{m\times n}$
and $\mathbf{x}\left(z\right)\in\left(\mathbb{C}^{\mathbb{C}}\right)^{n\times1}$
with entries given by
\begin{equation}
\mathbf{A}\left(z\right)\left[i,j\right]=\frac{z^{a_{ij}}}{\sqrt[n]{b_{i}}}\ \text{ and }\ \mathbf{x}\left(z\right)\left[j\right]=0\,z+x_{j}z^{0},\;\forall\,\begin{cases}
\begin{array}{c}
0\le i<m\\
0\le j<n
\end{array} & .\end{cases}
\end{equation}
The second type of systems of equations are obtained by setting the
\emph{combinator} to $\prod$ and are canonically prescribed by constraints
of the form
\begin{equation}
\mathbf{1}_{m\times1}=\mbox{CProd}_{\prod}\left(\mathbf{A}\left(z\right),\mathbf{x}\left(z\right)\right).
\end{equation}
 For example, let the exponent and variable constructs be given by
\begin{equation}
\mathbf{A}\left(z\right)=\left(\begin{array}{ccc}
\frac{z^{a_{00}}}{\sqrt{b_{0}}} &  & \frac{z^{a_{01}}}{\sqrt{b_{0}}}\\
\\
\frac{z^{a_{10}}}{\sqrt{b_{1}}} &  & \frac{z^{a_{11}}}{\sqrt{b_{1}}}
\end{array}\right)\ \text{ and }\ \mathbf{x}\left(z\right)=\left(\begin{array}{c}
0\,z+x_{0}z^{0}\\
\\
0\,z+x_{1}z^{0}
\end{array}\right).
\end{equation}
The corresponding \emph{system of second type} in the unknowns $x_{0}$,
$x_{1}$ expressed as 
\begin{equation}
\mathbf{1}_{2\times1}=\mbox{CProd}_{\prod}\left(\mathbf{A}\left(z\right),\mathbf{x}\left(z\right)\right)=\left(\begin{array}{r}
\frac{x_{0}^{a_{00}}\cdot x_{1}^{a_{01}}}{b_{0}}\\
\\
\frac{x_{0}^{a_{10}}\cdot x_{1}^{a_{11}}}{b_{1}}
\end{array}\right),
\end{equation}
where $\mathbf{A}\left(z\right)$ may be multivalued. Furthermore,
the entries of $\mathbf{x}$ need not be distinct variables. In particular
when entries of $\mathbf{x}$ lie in some non-Abelian group, systems
of equations of the second type describe instances of the word problem.
Note that special instances of the word problem are known to be undecidable.\\
\\
System of equations of the second type whose solutions lie $\mathbb{C}$
can be solved by elimination method very similar to Gaussian elimination.
Consider the system of two equations in the unknowns $x_{0}$, $x_{1}$
expressed in terms of exponent and variable construct
\begin{equation}
\mathbf{A}\left(z\right)=\left(\begin{array}{cc}
\frac{z^{a_{00}}}{\sqrt{b_{0}}} & \frac{z^{a_{01}}}{\sqrt{b_{0}}}\\
\frac{z^{a_{10}}}{\sqrt{b_{1}}} & \frac{z^{a_{11}}}{\sqrt{b_{1}}}
\end{array}\right),\ \mathbf{x}\left(z\right)=\left(\begin{array}{c}
x_{0}\\
x_{1}
\end{array}\right).
\end{equation}
given by 
\begin{equation}
\mathbf{1}_{2\times1}=\mbox{CProd}_{\prod}\left(\mathbf{A}\left(z\right),\mathbf{x}\left(z\right)\right)\Leftrightarrow\begin{cases}
\begin{array}{cccccc}
1= & x_{0}^{a_{00}} & \cdot & x_{1}^{a_{01}} & \cdot & b_{0}^{-1}\\
1= & x_{0}^{a_{10}} & \cdot & x_{1}^{a_{11}} & \cdot & b_{1}^{-1}
\end{array} & .\end{cases}\label{System of the second type}
\end{equation}
The system is put in REF via the row log-linear combination operation
\begin{equation}
\left(\text{R}_{0}^{a_{00}^{-1}}\right)^{-a_{10}}\text{R}_{1}\rightarrow\text{R}_{1}
\end{equation}
which yields 
\begin{equation}
\forall\:k\in\mathbb{Z},\quad\begin{cases}
\begin{array}{ccccccc}
1 & = & x_{0}^{a_{00}} & \cdot & x_{1}^{a_{01}} & \cdot & b_{0}^{-1}\\
1 & = & \left(x_{0}\right)^{\left(-a_{10}a_{00}^{-1}a_{00}+a_{10}\right)} & \cdot & \left(x_{1}\right)^{\left(-a_{10}a_{00}^{-1}a_{01}+a_{11}\right)} & \cdot & \left(\left(\left(b_{0}e^{i2\pi k}\right)^{a_{00}^{-1}}\right)^{-a_{10}}b_{1}\right)^{-1}
\end{array} & ,\end{cases}
\end{equation}
\begin{equation}
\implies\forall\:k\in\mathbb{Z},\quad\begin{cases}
\begin{array}{ccccccc}
1 & = & x_{0}^{a_{00}} & \cdot & x_{1}^{a_{01}} & \cdot & b_{0}^{-1}\\
1 & = & \left(x_{0}\right)^{0} & \cdot & \left(x_{1}\right)^{\left(-a_{10}a_{00}^{-1}a_{01}+a_{11}\right)} & \cdot & \left(\left(\left(b_{0}e^{i2\pi k}\right)^{a_{00}^{-1}}\right)^{-a_{10}}b_{1}\right)^{-1}
\end{array} & .\end{cases}
\end{equation}
The row log-linear combination operation therefore effects the change
\begin{equation}
\mathbf{A}\left(z\right)\begin{array}{c}
\longrightarrow\\
_{_{_{\left(\text{R}_{0}^{a_{00}^{-1}}\right)^{-a_{10}}\text{R}_{1}\rightarrow\text{R}_{1}}}}
\end{array}\mathbf{A}_{0}\left(z\right)=\left(\begin{array}{ccc}
\frac{z^{a_{00}}}{\sqrt{b_{0}}} &  & \frac{z^{a_{01}}}{\sqrt{b_{0}}}\\
\\
\frac{z^{0}}{\sqrt{\left(\left(b_{0}e^{i2\pi k}\right)^{a_{00}^{-1}}\right)^{-a_{10}}b_{1}}} &  & \frac{z^{\left(-a_{10}a_{00}^{-1}a_{01}+a_{11}\right)}}{\sqrt{\left(\left(b_{0}e^{i2\pi k}\right)^{a_{00}^{-1}}\right)^{-a_{10}}b_{1}}}
\end{array}\right),
\end{equation}
The following change of variable will allow us to transform the pivots
to 1. 
\begin{equation}
\mathbf{x}\left(z\right)=\left(\begin{array}{c}
\left(x_{0}^{a_{00}}\right)^{a_{00}^{-1}}\\
\left(x_{1}^{\left(-a_{10}a_{00}^{-1}a_{01}+a_{11}\right)}\right)^{\left(-a_{10}a_{00}^{-1}a_{01}+a_{11}\right)^{-1}}
\end{array}\right)=\left(\begin{array}{c}
y_{0}^{a_{00}^{-1}}\\
y_{1}^{\left(-a_{10}a_{00}^{-1}a_{01}+a_{11}\right)^{-1}}
\end{array}\right)=\mathbf{y}^{*}\left(z\right).
\end{equation}
The original system can thus be re-written as 
\begin{equation}
\mathbf{1}_{2\times1}=\mbox{CProd}_{\prod}\left(\mathbf{A}_{0}\left(z\right),\mathbf{y}^{*}\left(z\right)\right).
\end{equation}
Finally, the system is put in RREF via the row log-linear combination
operation 
\begin{equation}
\mathbf{A}_{0}\left(z\right)\begin{array}{c}
\longrightarrow\\
_{_{_{\text{R}_{1}^{-\left(\left(-a_{10}a_{00}^{-1}a_{01}+a_{11}\right)^{-1}a_{10}\right)}\cdot\text{R}_{0}\rightarrow\text{R}_{0}}}}
\end{array}\mathbf{A}_{1}\left(z\right)=\left(\begin{array}{rr}
\frac{z}{\sqrt{d_{0}}} & \frac{z^{0}}{\sqrt{d_{0}}}\\
\frac{z^{0}}{\sqrt{d_{1}}} & \frac{z^{1}}{\sqrt{d_{1}}}
\end{array}\right)
\end{equation}
where 
\begin{equation}
d_{0}=\left(\left(\left(\left(b_{0}e^{i2\pi k_{0}}\right)^{a_{00}^{-1}}\right)^{-a_{10}}b_{1}\right)^{-1}e^{i2\pi k_{1}}\right)^{-\left(-a_{10}a_{00}^{-1}a_{01}+a_{11}\right)^{-1}a_{01}}b_{0}^{-1},
\end{equation}
\begin{equation}
d_{1}=\left(\left(\left(b_{0}e^{i2\pi k_{0}}\right)^{a_{00}^{-1}}\right)^{-a_{10}}b_{1}\right)^{-1}
\end{equation}
The solution to the equations can be read from the RREF as $\left(\begin{array}{c}
y_{0}\\
y_{1}
\end{array}\right)=\left(\begin{array}{c}
d_{0}\\
d_{1}
\end{array}\right)$. The original unknown variables can be expressed as $\mathbf{x}\left(z\right)=\left(\begin{array}{c}
d_{0}^{a_{00}^{-1}}\\
d_{1}^{\left(-a_{10}a_{00}^{-1}a_{01}+a_{11}\right)^{-1}}
\end{array}\right)$ . More generally, a system of equations of the second type are canonically
expressed in terms of an exponent construct $\mathbf{A}\left(z\right)\in\left(\mathbb{C}^{\mathbb{C}}\right)^{m\times n}$
and a variable construct $\mathbf{x}\left(z\right)$ of size $n\times1$
with entries given by
\begin{equation}
\mathbf{A}\left(z\right)\left[i,j\right]=\frac{z^{a_{ij}}}{\sqrt[n]{b_{i}}},\;\mathbf{x}\left(z\right)\left[j\right]=0\,z+x_{j},\;\forall\:\begin{array}{c}
0\le i<m\\
0\le j<n
\end{array}.
\end{equation}
The corresponding system is of the form
\begin{equation}
\mathbf{1}_{m\times1}=\mbox{CProd}_{\prod}\left(\mathbf{A}\left(z\right),\mathbf{x}\left(z\right)\right).
\end{equation}
Solutions expressed as radical expressions of the coefficients of
$z$ in the construct $\mathbf{A}\left(z\right)$ are determined by
combining the following three fundamental row operations. The first
fundamental row operations are row log-linear combinations, specified
as follows 
\begin{equation}
\left(\text{R}_{i}\right)^{\alpha}\cdot\text{R}_{j}\rightarrow\text{R}_{j},
\end{equation}
for some non-zero $\alpha\in\mathbb{C}$. Note that row log-linear
combination may be multivalued. The second fundamental row operations
are row exchanges specified for $i\ne j$ by
\begin{equation}
\text{R}_{i}\leftrightarrow\text{R}_{j}.
\end{equation}
The third fundamental row operations are row log-scaling which are
obtained by variable change in $\mathbf{x}\left(z\right)$ of the
form 
\begin{equation}
x_{i}=\left(x_{i}^{\alpha}\right)^{\alpha^{-1}}=y_{i}^{\alpha^{-1}}=y_{i}^{*}.
\end{equation}
Finally, the third kind of system of equations is a variant of the
second kind in the sense that in both cases the combinator is set
to $\prod$. Systems of equations of the third kind are expressed
in terms are expressed in terms of a base exponent construct $\mathbf{A}\left(z\right)\in\left(\mathbb{C}^{\mathbb{C}}\right)^{m\times n}$
and a variable construct $\mathbf{x}\left(z\right)$ of size $n\times1$
whose entries are
\[
\mathbf{A}\left(z\right)\left[i,j\right]=\frac{a_{ij}^{z}}{\sqrt[n]{b_{i}}},\;\mathbf{x}\left(z\right)\left[j\right]=0z+x_{j},\;\forall\:\begin{array}{c}
0\le i<m\\
0\le j<n
\end{array}.
\]
The corresponding system is 
\begin{equation}
\mathbf{1}_{m\times1}=\mbox{CProd}_{\prod}\left(\mathbf{A}\left(z\right),\mathbf{x}\left(z\right)\right).
\end{equation}
Solutions of system of equations of the third kind are obtained by
combining the following three fundamental row operations used to derive
solutions to systems of equations of the second kind. We illustrate
the method of elimination on systems of the third type. Consider the
system of two equations in the unknowns $x_{0}$, $x_{1}$ expressed
in terms of 
\begin{equation}
\mathbf{A}\left(z\right)=\left(\begin{array}{cc}
\frac{a_{00}^{z}}{\sqrt{b_{0}}} & \frac{a_{01}^{z}}{\sqrt{b_{0}}}\\
\frac{a_{10}^{z}}{\sqrt{b_{1}}} & \frac{a_{11}^{z}}{\sqrt{b_{1}}}
\end{array}\right),\ \mathbf{x}\left(z\right)=\left(\begin{array}{c}
x_{0}\\
x_{1}
\end{array}\right),
\end{equation}
given by 
\begin{equation}
\mathbf{1}_{2\times1}=\mbox{CProd}_{\prod}\left(\mathbf{A}\left(z\right),\mathbf{B}\left(z\right)\right)
\end{equation}
more explicitly written as 
\begin{equation}
\begin{cases}
\begin{array}{ccccccc}
1 & = & a_{00}^{x_{0}} & \cdot & a_{01}^{x_{1}} & \cdot & b_{0}^{-1}\\
1 & = & a_{10}^{x_{0}} & \cdot & a_{11}^{x_{1}} & \cdot & b_{1}^{-1}
\end{array}\end{cases}.
\end{equation}
The row log-linear combination 
\begin{equation}
\text{R}_{0}^{-\frac{\ln a_{10}}{\ln a_{00}}}\cdot\text{R}_{1}\rightarrow\text{R}_{1}
\end{equation}
yields 
\begin{equation}
\forall\:k\in\mathbb{Z},\quad\begin{cases}
\begin{array}{ccccccc}
1 & = & a_{00}^{x_{0}} & \cdot & a_{01}^{x_{1}} & \cdot & b_{0}^{-1}\\
1 & = & \left(a_{00}^{-\frac{\ln a_{10}}{\ln a_{00}}}a_{10}\right)^{x_{0}} & \cdot & \left(a_{01}^{-\frac{\ln a_{10}}{\ln a_{00}}}a_{11}\right)^{x_{1}} & \cdot & \left(\left(b_{0}e^{i2\pi k}\right)^{-\frac{\ln a_{10}}{\ln a_{00}}}b_{1}\right)^{-1}
\end{array}\end{cases},
\end{equation}
\begin{equation}
\forall\:k\in\mathbb{Z},\quad\begin{cases}
\begin{array}{ccccccc}
1 & = & a_{00}^{x_{0}} & \cdot & a_{01}^{x_{1}} & \cdot & b_{0}^{-1}\\
1 & = & \left(a_{10}^{0}\right)^{x_{0}} & \cdot & \left(a_{01}^{-\frac{\ln a_{10}}{\ln a_{00}}}a_{11}\right)^{x_{1}} & \cdot & \left(\left(b_{0}e^{i2\pi k}\right)^{-\frac{\ln a_{10}}{\ln a_{00}}}b_{1}\right)^{-1}
\end{array}\end{cases}.
\end{equation}
The row log-linear combination operation therefore effects 
\begin{equation}
\mathbf{A}\left(z\right)\begin{array}{c}
\longrightarrow\\
_{_{_{\text{R}_{0}^{-\frac{\ln a_{10}}{\ln a_{00}}}\cdot\text{R}_{1}\rightarrow\text{R}_{1}}}}
\end{array}\mathbf{A}_{0}\left(z\right)=\left(\begin{array}{ccc}
\frac{a_{00}^{z}}{\sqrt{b_{0}}} &  & \frac{a_{01}^{z}}{\sqrt{b_{0}}}\\
\\
\frac{1^{z}}{\sqrt{\left(b_{0}e^{i2\pi k}\right)^{-\frac{\ln a_{10}}{\ln a_{00}}}b_{1}}} &  & \frac{\left(a_{01}^{-\frac{\ln a_{10}}{\ln a_{00}}}a_{11}\right)^{z}}{\sqrt{\left(b_{0}e^{i2\pi k}\right)^{-\frac{\ln a_{10}}{\ln a_{00}}}b_{1}}}
\end{array}\right).
\end{equation}
The following change of variable will allow us to transform the pivots
to 1. 
\begin{equation}
\mathbf{x}\left(z\right)=\left(\begin{array}{c}
\exp\left\{ x_{0}\left(\ln a_{00}\right)^{2}\frac{1}{\ln a_{00}}\right\} \\
\exp\left\{ x_{1}\ln\left(a_{01}^{-\frac{\ln a_{10}}{\ln a_{00}}}a_{11}\right)^{2}\frac{1}{\ln\left(a_{01}^{-\frac{\ln a_{10}}{\ln a_{00}}}a_{11}\right)}\right\} 
\end{array}\right)=\left(\begin{array}{c}
\exp\left\{ \frac{y_{0}}{\ln a_{00}}\right\} \\
\exp\left\{ \frac{y_{1}}{\ln\left(a_{01}^{-\frac{\ln a_{10}}{\ln a_{00}}}a_{11}\right)}\right\} 
\end{array}\right)=\mathbf{y}^{*}\left(z\right)
\end{equation}
The original system can thus be re-written as 
\begin{equation}
\mathbf{1}_{2\times1}=\mbox{CProd}_{\prod}\left(\mathbf{A}_{0}\left(z\right),\mathbf{y}^{*}\left(z\right)\right).
\end{equation}
The system is put in RREF via the row log-linear combination operation
\begin{equation}
\text{R}_{1}^{\frac{-\ln a_{01}}{\ln\left(a_{01}^{-\frac{\ln a_{10}}{\ln a_{00}}}a_{11}\right)}}\cdot\text{R}_{0}\rightarrow\text{R}_{0}
\end{equation}
\begin{equation}
\forall\:k_{0},k_{1}\in\mathbb{Z},\quad\begin{cases}
\begin{array}{ccccccc}
1 & = & e^{y_{0}} & \cdot & 1^{y_{1}} & \cdot & \left(\left(\left(b_{0}e^{i2\pi k_{0}}\right)^{-\frac{\ln a_{10}}{\ln a_{00}}}b_{1}e^{i2\pi k_{1}}\right)^{\frac{-\ln a_{01}}{\ln\left(a_{01}^{-\frac{\ln a_{10}}{\ln a_{00}}}a_{11}\right)}}b_{0}\right)^{-1}\\
1 & = & 1^{y_{0}} & \cdot & e^{y_{1}} & \cdot & \left(\left(b_{0}e^{i2\pi k_{0}}\right)^{-\frac{\ln a_{10}}{\ln a_{00}}}b_{1}\right)^{-1}
\end{array}\end{cases}.
\end{equation}
Equivalently
\begin{equation}
\mathbf{A}_{0}\left(z\right)\begin{array}{c}
\longrightarrow\\
_{_{_{\text{R}_{1}^{\frac{-\ln a_{01}}{\ln\left(a_{01}^{-\frac{\ln a_{10}}{\ln a_{00}}}a_{11}\right)}}\cdot\text{R}_{0}\rightarrow\text{R}_{0}}}}
\end{array}\mathbf{A}_{1}\left(z\right)=\left(\begin{array}{rr}
\frac{e^{z}}{\sqrt{d_{0}}} & \frac{e^{0z}}{\sqrt{d_{0}}}\\
\frac{e^{0z}}{\sqrt{d_{1}}} & \frac{e^{z}}{\sqrt{d_{1}}}
\end{array}\right)
\end{equation}
where 
\[
d_{0}=\left(\left(\left(b_{0}e^{i2\pi k_{0}}\right)^{-\frac{\ln a_{10}}{\ln a_{00}}}b_{1}e^{i2\pi k_{1}}\right)^{\frac{-\ln a_{01}}{\ln\left(a_{01}^{-\frac{\ln a_{10}}{\ln a_{00}}}a_{11}\right)}}b_{0}\right)^{-1},
\]
\[
d_{1}=\left(\left(b_{0}e^{i2\pi k}\right)^{-\frac{\ln a_{10}}{\ln a_{00}}}b_{1}\right)^{-1}.
\]
The solution to the equations can be read from the RREF as $\left(\begin{array}{c}
y_{0}\\
y_{1}
\end{array}\right)=\left(\begin{array}{c}
\text{ln}d_{0}\\
\text{ln}d_{1}
\end{array}\right)$. We omit here the cumbersome explicit expressions of entries $\mathbf{x}\left(z\right)$.
In summary, the differences between the types of systems of systems
is predicated their canonical formulation which expressed either in
terms of a coefficient, an exponent or a base constructs. All three
of which are respectively illustrated by $2\times2$ constructs
\begin{equation}
\begin{array}{ccc}
\mathbf{A}\left(z\right) & = & \left(\begin{array}{rr}
a_{00}z-\frac{b_{0}}{2} & a_{01}z-\frac{b_{0}}{2}\\
a_{10}z-\frac{b_{1}}{2} & a_{11}z-\frac{b_{1}}{2}
\end{array}\right),\\
\\
\mathbf{Ba}\left(z\right) & = & \left(\begin{array}{rr}
\frac{z^{a_{00}}}{\sqrt{b_{0}}} & \frac{z^{a_{01}}}{\sqrt{b_{0}}}\\
\frac{z^{a_{10}}}{\sqrt{b_{1}}} & \frac{z^{a_{11}}}{\sqrt{b_{1}}}
\end{array}\right),\\
 & \text{and}\\
\mathbf{Ca}\left(z\right) & = & \left(\begin{array}{rr}
\frac{a_{00}^{z}}{\sqrt{b_{0}}} & \frac{a_{01}^{z}}{\sqrt{b_{0}}}\\
\frac{a_{10}^{z}}{\sqrt{b_{1}}} & \frac{a_{11}^{z}}{\sqrt{b_{1}}}
\end{array}\right).
\end{array}.
\end{equation}
Each one is used to express a different types of systems of equations
as illustrated by the SageMath code below\\
\begin{sageverbatim}

sage: # Loading the Package into SageMath
sage: load('./Hypermatrix_Algebra_tst.sage')
sage: 
sage: # Initialization of the variables
sage: z=var('z')
sage: 
sage: # Initialization of the order and size parameter
sage: od=2; sz=2
sage: 
sage: # Initialization of the constructs
sage: Lb=var_list('b',sz)
sage: A=z*HM(sz,sz,'a')-HM(2,1,[Lb[0]/2,Lb[1]/2])*HM(1,2,'one')
sage: X=HM(sz,1,var_list('x',sz))
sage: 
sage: # Computing the product associated with systems of the first type
sage: C=GProd([A,X], sum, [z])
sage: C
[[a00*x0 + a01*x1 - b0], [a10*x0 + a11*x1 - b1]]
sage: # Initialization of the construct
sage: Ba=(HM(sz,sz,'a').elementwise_base_exponent(z)).elementwise_product(\
....: HM(2,1,[Lb[0]^(-1/2), Lb[1]^(-1/2)])*HM(1,2,'one'))
sage: Ba
[[z^a00/sqrt(b0), z^a01/sqrt(b0)], [z^a10/sqrt(b1), z^a11/sqrt(b1)]]
sage: # Computing the product associated with systems of the second type
sage: Bc=GProd([Ba,X], prod, [z])
sage: Bc
[[x0^a00*x1^a01/b0], [x0^a10*x1^a11/b1]]
sage: # Initialization of the construct
sage: Ca=(HM(sz,sz,'a').elementwise_exponent(z)).elementwise_product(\
....: HM(2,1,[Lb[0]^(-1/2), Lb[1]^(-1/2)])*HM(1,2,'one'))
sage: Ca
[[a00^z/sqrt(b0), a01^z/sqrt(b0)], [a10^z/sqrt(b1), a11^z/sqrt(b1)]]
sage: # Computing the product associated with systems of the third type
sage: Cc=GProd([Ca,X], prod, [z])
sage: Cc
[[a00^x0*a01^x1/b0], [a10^x0*a11^x1/b1]]
\end{sageverbatim}\\
An illustration for each one of the corresponding three types of systems
are\begin{sagesilent}

# Loading the Package
load('./Hypermatrix_Algebra_tst.sage')

# Initialization of the variables
z=var('z')

# Initialization of the order and size parameter
od=2; sz=2

# Initialization of the constructs
Lb=var_list('b',sz)
A=z*HM(sz,sz,'a')-HM(2,1,[Lb[0]/2,Lb[1]/2])*HM(1,2,'one')
X=HM(sz,1,var_list('x',sz))

# Computing the product associated with systems of the first type
C=GProd([A,X], sum, [z])

# Initialization of the construct
Ba=(HM(sz,sz,'a').elementwise_base_exponent(z)).elementwise_product(\
HM(2,1,[Lb[0]^(-1/2), Lb[1]^(-1/2)])*HM(1,2,'one'))

# Computing the product associated with systems of the second type
Bc=GProd([Ba,X], prod, [z])

# Initialization of the construct
Ca=(HM(sz,sz,'a').elementwise_exponent(z)).elementwise_product(\
HM(2,1,[Lb[0]^(-1/2), Lb[1]^(-1/2)])*HM(1,2,'one'))

# Computing the product associated with systems of the third type
Cc=GProd([Ca,X], prod, [z])

\end{sagesilent}
\begin{equation}
\begin{array}{ccc}
\mathbf{0}_{2\times1}=\text{CProd}_{\sum}\left(\mathbf{A}\left(z\right),\mathbf{x}\left(z\right)\right) & = & \sage{C.matrix()},\\
\\
\mathbf{1}_{2\times1}=\text{CProd}_{\prod}\left(\mathbf{Ba}\left(z\right),\mathbf{x}\left(z\right)\right) & = & \sage{Bc.matrix()}\\
\\
\mathbf{1}_{2\times1}=\text{CProd}_{\prod}\left(\mathbf{Ca}\left(z\right),\mathbf{x}\left(z\right)\right) & = & \sage{Cc.matrix()}.
\end{array},
\end{equation}
The variable construct across all three types is same and in our illustration
given by $\mathbf{x}\left(z\right)=\left(\begin{array}{r}
x_{0}\\
x_{1}
\end{array}\right).$ For a given system of one of the three types we illustrated by examples
how such a system is put in RREF. We now extend the notion of RREF
to arbitrary systems of equations. For this purpose we define the
degree matrix of a system. We express an arbitrary systems of equations
using constructs $\mathbf{A}\left(z\right)\in\left(\mathbb{K}^{\mathbb{K}}\right)^{m\times n}$
and $\mathbf{x}\left(z\right)$ as follows 
\begin{equation}
\mathbf{1}_{m\times1}=\mbox{GProd}_{\sum}\left(\mathbf{A}\left(z\right),\mathbf{x}\left(z\right)\right)
\end{equation}
The degree matrix of such a system noted $\mathcal{D}\left(\mathbf{A}\left(z\right)\right)$
is an $m\times n$ matrix whose entries are given by 
\begin{equation}
\mathcal{D}\left(\mathbf{A}\left(z\right)\right)\left[i,j\right]=\text{degree of }z\text{ in }\mathbf{A}\left(z\right)\left[i,j\right],\;\forall\:\begin{array}{c}
0\le i<m\\
0\le j<n
\end{array}.
\end{equation}
We say that a system is REF or RREF if its degree matrix is in REF
or RREF respectively. Elimination methods proceed by combining a pre-defined
set of fundamental row operation to put the system in RREF. It is
not uncommon to devise system of equation by composing other systems
of equations. For instance every system of algebraic equation is canonically
expressed as a composition of a system of the second and first kind
of the form
\begin{equation}
\mathbf{0}_{n\times1}=\text{CProd}_{\sum}\left(\mathbf{B}\left(z\right),\text{CProd}_{\prod}\left(\mathbf{A}\left(z\right),\mathbf{x}\left(z\right)\right)\right),
\end{equation}
where $\mathbf{A}\left(z\right)\in\left(\mathbb{C}^{\mathbb{C}}\right)^{m\times\ell}$,
$\mathbf{B}\left(z\right)\in\left(\mathbb{C}^{\mathbb{C}}\right)^{n\times m}$
and $\mathbf{x}\left(z\right)$ of size $\ell\times1$ denote respectively
an exponent, coefficient and variable constructs with entries given
by 
\begin{equation}
\mathbf{A}\left(z\right)\left[i,t\right]=\frac{z^{a_{it}}}{\sqrt[n]{b_{i}}},\;\mathbf{B}\left(z\right)\left[j,k\right]=\alpha_{ij}\,z-\frac{\beta_{i}}{n}\text{ and }\;\mathbf{x}\left(z\right)\left[t\right]=0\,z+x_{t},\;\forall\:\begin{array}{c}
0\le i,k<m\\
0\le t<\ell\\
0\le j<n
\end{array}.
\end{equation}
\[
,
\]
For notational convenience we express the canonical form of a system
of algebraic equation as 
\begin{equation}
\mathbf{0}_{n\times1}=\text{CProd}_{\sum}\left(\mathbf{B}\left(z\right),\text{CProd}_{\prod}\left(\mathbf{A}\left(z\right),\mathbf{x}\left(z\right)\right)\right)\Leftrightarrow\mathbf{0}_{n\times1}=\mathbf{C}\left(\mathbf{x}^{\mathbf{E}}\right)
\end{equation}
where $\mathbf{E}\in\mathbb{Z}^{m\times\ell}$ is called the exponent
matrix of the system of algebraic equation and $\mathbf{C}\in\mathbb{C}^{n\times m}$
is called the coefficient matrix of the system where 
\[
\mathbf{C}=\mathbf{B}\left(1\right).
\]

We conclude by discussing constraints of the form
\begin{equation}
\mathbf{1}_{m\times1}=\mbox{CProd}_{\prod}\left(\mathbf{A}\left(z\right),\mathbf{x}\left(z\right)\right),\quad\mathbf{x}\left(z\right)\left[j\right]=0\,z+x_{j},\;\forall\:0\le j<n
\end{equation}
where the construct $\mathbf{A}\left(z\right)\in\left(\mathbb{C}^{\mathbb{C}}\right)^{m\times n}$
whose are either given by 
\begin{equation}
\mathbf{A}\left(z\right)\left[i,j\right]=\frac{z^{a_{ij}}}{\sqrt[n]{b_{i}}},\;\forall\:\begin{array}{c}
0\le i<m\\
0\le j<n
\end{array},\text{ or alternatively }\mathbf{A}\left(z\right)\left[i,j\right]=\frac{a_{ij}^{z}}{\sqrt[n]{b_{i}}},\;\forall\:\begin{array}{c}
0\le i<m\\
0\le j<n
\end{array},
\end{equation}
( associated with systems of the second and third type respectively
) where $\left|\left\{ b_{i}\right\} _{0\le i<m}\cap\left\{ 0\right\} \right|\ge0$.
On the one hand, allowing for $\left|\left\{ b_{i}\right\} _{0\le i<m}\cap\left\{ 0\right\} \right|>0$
in systems of second type yields for a subset of the constrains of
the form 
\begin{equation}
\frac{0}{0}=\left(\prod_{0\le j<n}\left(\mathbf{x}\left[j\right]\right)^{\mathbf{A}\left[i,j\right]}\right)\cdot0^{-1}.
\end{equation}
Any such constraint is of course meaningless unless a factor of $\underset{0\le j<n}{\prod}\left(\mathbf{x}\left[j\right]\right)^{\mathbf{A}\left[i,j\right]}$
also equals zero. Choices among the $2^{n}$ possible factors to be
set to zero lead to a combinatorial branching of options which lies
at the heart of computational intractability. On the other hand, allowing
for $\left|\left\{ b_{i}\right\} _{0\le i<m}\cap\left\{ 0\right\} \right|>0$
in systems of third type leads to the non existence of bounded solutions
to such system.

\subsection{An interpolation perspective to systems of equations}

To motivate a perspective different from the Gaussian elimination
approach to the art of solving systems of equations of the first type
we briefly review the Lagrange interpolation construction which determines
the minimal degree polynomial $f\left(x\right)\in\mathbb{C}\left[x\right]$
subject to the constraints 
\[
S\,:=\left\{ \left(x=a_{i},\,f\left(a_{i}\right)=b_{i}\right)\,:\,i\in\left[0,n\right)\cap\mathbb{Z}\right\} ,
\]
where $\left|\left\{ a_{i}\,:\,i\in\left[0,n\right)\cap\mathbb{Z}\right\} \right|=n$.
\begin{equation}
f\left(x\right)=\sum_{i\in\left[0,n\right)\cap\mathbb{Z}}b_{i}\,\left(\prod_{\begin{array}{c}
0\le s<t<n\\
i\notin\left\{ s,t\right\} 
\end{array}}\left(a_{t}-a_{s}\right)\prod_{j\in\left[0,n\right)\cap\mathbb{Z}\backslash\left\{ i\right\} }\left(x-a_{j}\right)\right)\left(\prod_{0\le u<v<n}\left(a_{v}-a_{u}\right)\right)^{-1},
\end{equation}
which we rewrite as
\begin{equation}
f\left(x\right)=\sum_{i\in\left[0,n\right)\cap\mathbb{Z}}b_{i}\,\left(\prod_{j\in\left[0,n\right)\cap\mathbb{Z}\backslash\left\{ i\right\} }\left(x-a_{j}\right)\right)\left(\prod_{j\in\left[0,n\right)\cap\mathbb{Z}\backslash\left\{ i\right\} }\left(a_{i}-a_{j}\right)\right)^{-1}.
\end{equation}
It is well known that the solution to interpolation instances as specified
above reduces to solving a system of algebraic equations whose coefficient
matrix is a Vandermonde matrix. We argue here that in some sense the
converse also holds with some minor caveats. Solving systems of equations
of the first type is equivalent to extending Lagrange's polynomial
interpolating construction to vector inputs as follows 
\[
f\,:\,\mathbb{C}^{1\times n}\rightarrow\mathbb{C},
\]
such that 
\[
S\,:=\left\{ \left(\mathbf{x}=\mathbf{A}\left[i,:\right],\,f\left(\mathbf{A}\left[i,:\right]\right)=b_{i}\right)\,:\,i\in\left[0,n\right)\cap\mathbb{Z}\right\} ,
\]
where $\mathbf{A}\in$ GL$_{n}\left(\mathbb{C}\right)$. Modifying
accordingly Lagrange's interpolating construction, we write
\[
f\left(\mathbf{x}\right)=\sum_{i\in\left[0,n\right)\cap\mathbb{Z}}b_{i}\left(\left(\underset{\begin{array}{c}
0\le s<t<n\\
i\notin\left\{ s,t\right\} 
\end{array}}{\bigcirc}\left(\mathbf{A}\left[t,:\right]-\mathbf{A}\left[s,:\right]\right)\right)\circ\left(\underset{_{j\in\left[0,n\right)\cap\mathbb{Z}\backslash\left\{ i\right\} }}{\bigcirc}\left(\mathbf{x}-\mathbf{A}\left[j,:\right]\right)\right)\right)\times
\]
\[
\left(\text{diag}\left(\underset{_{0\le u<v<n}}{\bigcirc}\left(\mathbf{A}\left[v,:\right]-\mathbf{A}\left[u,:\right]\right)\right)^{+}\cdot\left(\frac{\mathbf{w}^{\circ^{0}}}{n}+\sum_{0<k<n}\gamma_{k}\,\mathbf{w}^{\circ^{k}}\right)\right)^{\top}
\]
which we rewrite as
\[
f\left(\mathbf{x}\right)=\sum_{i\in\left[0,n\right)\cap\mathbb{Z}}b_{i}\,\left(\underset{_{j\in\left[0,n\right)\cap\mathbb{Z}\backslash\left\{ i\right\} }}{\bigcirc}\left(\mathbf{x}-\mathbf{A}\left[j,:\right]\right)\right)\cdot\left(\text{diag}\left(\underset{_{j\in\left[0,n\right)\cap\mathbb{Z}\backslash\left\{ i\right\} }}{\bigcirc}\left(\mathbf{A}\left[i,:\right]-\mathbf{A}\left[j,:\right]\right)\right)^{+}\left(\frac{\mathbf{w}^{\circ^{0}}}{n}+\sum_{0<k<n}\gamma_{k}\,\mathbf{w}^{\circ^{k}}\right)\right)^{\top}
\]
In expanded form we have 
\begin{equation}
f\left(\mathbf{x}\right)=\sum_{0\le k<n}\left(\mathbf{x}^{\circ^{k}}\right)^{\top}\cdot\mathbf{c}_{k}\left(\gamma_{1},\cdots,\gamma_{n-1}\right).
\end{equation}
Thus far the interpolation construction did not use the fact that
$f\left(\mathbf{x}\right)$ is linear i.e. for all $c\in\mathbb{C}$
and $\mathbf{u},\mathbf{v}\in\mathbb{C}^{1\times n}$ we have 
\begin{equation}
\begin{array}{ccc}
f\left(c\,\mathbf{u}\right) & = & c\,f\left(\mathbf{u}\right)\\
f\left(\mathbf{u}+\mathbf{v}\right) & = & f\left(\mathbf{u}\right)+f\left(\mathbf{v}\right)
\end{array}.
\end{equation}
Consequently the solution is determined by solving a smaller system
of linear equation equation. 
\begin{equation}
f\left(\mathbf{x}\right)=\mathbf{b}\cdot\mathbf{v}\left(\mathbf{A},\mathbf{x},\gamma_{1},\cdots,\gamma_{n-1}\right)\mod\left\{ \mathbf{c}_{k}\left(\gamma_{1},\cdots,\gamma_{n-1}\right)\right\} _{1<k<n}
\end{equation}
where
\begin{equation}
\mathbf{v}\left(\mathbf{A},\mathbf{x},\gamma_{1},\cdots,\gamma_{n-1}\right)\left[i\right]=\left(\underset{_{0\le j\ne i<n}}{\bigcirc}\left(\mathbf{x}-\mathbf{A}\left[j,:\right]\right)\right)\cdot\left(\text{diag}\left(\underset{_{0\le j\ne i<n}}{\bigcirc}\left(\mathbf{A}\left[i,:\right]-\mathbf{A}\left[j,:\right]\right)\right)^{+}\left(\frac{\mathbf{w}^{\circ^{0}}}{n}+\sum_{0<k<n}\gamma_{k}\,\mathbf{w}^{\circ^{k}}\right)\right)^{\top}
\end{equation}
The case $n=2$ therefore provides us with a base case for the inductive
argument. Consider symbolic matrices,
\[
\mathbf{A}=\left(\begin{array}{rr}
a_{00} & a_{01}\\
a_{10} & a_{11}
\end{array}\right),\quad\mathbf{b}=\left(\begin{array}{r}
b_{0}\\
b_{1}
\end{array}\right).
\]
As prescribed by the construction above 
\[
f\,:\,\mathbb{C}^{1\times2}\rightarrow\mathbb{C},
\]
\begin{equation}
f\left(\mathbf{x}\right)=b_{1}\,\left(\mathbf{x}-\mathbf{A}\left[0,:\right]\right)\left(\begin{array}{c}
\frac{\nicefrac{1}{2}+\gamma}{a_{10}-a_{00}}\\
\frac{\nicefrac{1}{2}+\gamma}{a_{11}-a_{01}}
\end{array}\right)+b_{0}\,\left(\mathbf{x}-\mathbf{A}\left[1,:\right]\right)\left(\begin{array}{c}
\frac{\nicefrac{1}{2}+\gamma}{a_{00}-a_{10}}\\
\frac{\nicefrac{1}{2}+\gamma}{a_{01}-a_{11}}
\end{array}\right)
\end{equation}
\begin{equation}
\implies f\left(\mathbf{x}\right)=\mathbf{x}\left[b_{1}\left(\begin{array}{c}
\frac{\nicefrac{1}{2}+\gamma}{a_{10}-a_{00}}\\
\frac{\nicefrac{1}{2}+\gamma}{a_{11}-a_{01}}
\end{array}\right)+b_{0}\left(\begin{array}{c}
\frac{\nicefrac{1}{2}+\gamma}{a_{00}-a_{10}}\\
\frac{\nicefrac{1}{2}+\gamma}{a_{01}-a_{11}}
\end{array}\right)\right]-\left[b_{1}\,\mathbf{A}\left[0,:\right]\left(\begin{array}{c}
\frac{\nicefrac{1}{2}+\gamma}{a_{10}-a_{00}}\\
\frac{\nicefrac{1}{2}-\gamma}{a_{11}-a_{01}}
\end{array}\right)+b_{0}\,\mathbf{A}\left[1,:\right]\left(\begin{array}{c}
\frac{\nicefrac{1}{2}+\gamma}{a_{00}-a_{10}}\\
\frac{\nicefrac{1}{2}-\gamma}{a_{01}-a_{11}}
\end{array}\right)\right]
\end{equation}
The solution is determined by 
\begin{equation}
f\left(\mathbf{x}\right)=\mathbf{x}\left[b_{1}\left(\begin{array}{c}
\frac{\nicefrac{1}{2}+\gamma}{a_{10}-a_{00}}\\
\frac{\nicefrac{1}{2}+\gamma}{a_{11}-a_{01}}
\end{array}\right)+b_{0}\left(\begin{array}{c}
\frac{\nicefrac{1}{2}+\gamma}{a_{00}-a_{10}}\\
\frac{\nicefrac{1}{2}+\gamma}{a_{01}-a_{11}}
\end{array}\right)\right]-\left[b_{1}\,\mathbf{A}\left[0,:\right]\left(\begin{array}{c}
\frac{\nicefrac{1}{2}+\gamma}{a_{10}-a_{00}}\\
\frac{\nicefrac{1}{2}-\gamma}{a_{11}-a_{01}}
\end{array}\right)+b_{0}\,\mathbf{A}\left[1,:\right]\left(\begin{array}{c}
\frac{\nicefrac{1}{2}+\gamma}{a_{00}-a_{10}}\\
\frac{\nicefrac{1}{2}-\gamma}{a_{01}-a_{11}}
\end{array}\right)\right]\text{ mod}f\left(\mathbf{0}_{1\times2}\right)
\end{equation}
\[
\implies b_{1}\,\mathbf{A}\left[0,:\right]\left(\begin{array}{c}
\frac{\nicefrac{1}{2}+\gamma}{a_{10}-a_{00}}\\
\frac{\nicefrac{1}{2}-\gamma}{a_{11}-a_{01}}
\end{array}\right)+b_{0}\,\mathbf{A}\left[1,:\right]\left(\begin{array}{c}
\frac{\nicefrac{1}{2}+\gamma}{a_{00}-a_{10}}\\
\frac{\nicefrac{1}{2}-\gamma}{a_{01}-a_{11}}
\end{array}\right)=0
\]
The SageMath code setup for the derivation in the above example is
as follows\\
\begin{sageverbatim}
sage: # Loading the Hypermatrix package
sage: load('./Hypermatrix_Algebra_tst.sage')
sage: 
sage: # Initialization of the variables
sage: gamma=var('gamma'); HM(2,2,'a')
[[a00, a01], [a10, a11]]
sage: var_list('b',2); var_list('x',2)
[b0, b1]
[x0, x1]
sage: # Initialization of the polynomial construction
sage: F=(\
....: b1*HM(1,2,[x0-a00,x1-a01])*HM(2,1,[(1/2+gamma)/(a10-a00), (1/2-gamma)/(a11-a01)])+\
....: b0*HM(1,2,[x0-a10,x1-a11])*HM(2,1,[(1/2+gamma)/(a00-a10), (1/2-gamma)/(a01-a11)]))[0,0]
sage: 
sage: # Solving for the value of gamma
sage: SlnF=solve(F.subs([x0==0,x1==0]),gamma)
sage: 
sage: # Initialization of the construction
sage: G=(F.subs(SlnF)).canonicalize_radical()
sage: G
-((a11*b0 - a01*b1)*x0 - (a10*b0 - a00*b1)*x1)/(a01*a10 - a00*a11)
\end{sageverbatim}\\
from which we have\begin{sagesilent}

# Loading the Hypermatrix package
load('./Hypermatrix_Algebra_tst.sage')

# Initialization of the variables
gamma=var('gamma'); HM(2,2,'a')
var_list('b',2); var_list('x',2)

F=(\
b1*HM(1,2,[x0-a00,x1-a01])*HM(2,1,[(1/2+gamma)/(a10-a00), (1/2-gamma)/(a11-a01)])+\
b0*HM(1,2,[x0-a10,x1-a11])*HM(2,1,[(1/2+gamma)/(a00-a10), (1/2-gamma)/(a01-a11)]))[0,0]

# Solving for the value of gamma
SlnF=solve(F.subs([x0==0,x1==0]),gamma)

# Initialization of the construction
G=(F.subs(SlnF)).canonicalize_radical()

\end{sagesilent}
\[
f\left(\mathbf{x}\right)=\sage{G}.
\]
Let us discuss the interpolation approach to solving systems of linear
equations when working over a finite field ( also called Galois field)
$\mathbb{F}_{p^{m}}$, for some prime number $p$. It is clear that
the interpolation construction previously discussed remains valid
when the ground field is taken to be $\mathbb{F}_{p^{m}}$. Note that
decimal encoding in base $p^{m}$ provides a canonical map from $\left(\mathbb{F}_{p^{m}}\right)^{1\times n}$
to integers $\left[0,p^{mn}\right)\cap\mathbb{Z}$ as follows 
\begin{equation}
\forall\,\mathbf{v}\in\left(\mathbb{F}_{p^{m}}\right)^{1\times n},\quad\text{lex}\left(\mathbf{v}\right)=\sum_{0\le i<n}v_{i}\,p^{m\,i}.
\end{equation}
Given this mapping, the interpolation problem 
\[
f\,:\,\left(\mathbb{F}_{p^{m}}\right)^{1\times n}\rightarrow\mathbb{F}_{p^{mn}},
\]
such that 
\[
S\,:=\left\{ \left(\mathbf{x}=\mathbf{A}\left[i,:\right],\,f\left(\mathbf{A}\left[i,:\right]\right)=b_{i}\right)\,:\,i\in\left[0,p^{mn}\right)\cap\mathbb{Z}\right\} .
\]
reduces to the univariate interpolation construction prescribed over
\[
S^{\prime}\,:=\left\{ \left(\text{lex}\left(\mathbf{x}\right)=\text{lex}\left(\mathbf{A}\left[i,:\right]\right),\,p\left(\text{lex}\left(\mathbf{A}\left[i,:\right]\right)\right)=b_{i}\right)\,:\,i\in\left[0,p^{mn}\right)\cap\mathbb{Z}\right\} ,
\]
\begin{equation}
g\left(\text{lex}\left(\mathbf{x}\right)\right)=\sum_{i\in\left[0,p^{mn}\right)\cap\mathbb{Z}}b_{i}\,\left(\prod_{j\in\left[0,p^{mn}\right)\cap\mathbb{Z}\backslash\left\{ i\right\} }\left(\text{lex}\left(\mathbf{x}\right)-\text{lex}\left(\mathbf{A}\left[j,:\right]\right)\right)\right)\left(\prod_{j\in\left[0,p^{mn}\right)\cap\mathbb{Z}\backslash\left\{ i\right\} }\left(\text{lex}\left(\mathbf{A}\left[i,:\right]\right)-\text{lex}\left(\mathbf{A}\left[j,:\right]\right)\right)\right)^{-1}.
\end{equation}
The linearity assumption determines the interpolating construction
by specifying at most $n$ interpolating points as opposed to $p^{mn}$
interpolating points. As a result the number of roots of $g\left(\text{lex}\left(\mathbf{x}\right)\right)$
in $\left[0,p^{mn}\right)\cap\mathbb{Z}$ is given by 
\[
\left(p^{m}\right)^{\text{dim}\left(\text{Null Space}\mathbf{A}\right)}.
\]
Note that have to be careful to work over the smallest field with
more then $p^{mn}$ elements. 

Consider the following variant of Lagrange's interpolation construction
for 
\[
S\,:=\left\{ \left(x=a_{i},\,f\left(a_{i}\right)=b_{i}\right)\,:\,i\in\left[0,n\right)\cap\mathbb{Z}\right\} ,
\]
a variant of the Lagrange interpolating construction is given by 
\[
f\left(x\right)=\prod_{i\in\left[0,n\right)\cap\mathbb{Z}}b_{i}^{\left(\underset{j\in\left[0,n\right)\cap\mathbb{Z}\backslash\left\{ i\right\} }{\prod}\left(x-a_{j}\right)\right)\left(\underset{j\in\left[0,n\right)\cap\mathbb{Z}\backslash\left\{ i\right\} }{\prod}\left(a_{i}-a_{j}\right)\right)^{-1}}.
\]
This new interpolation construction can be adapted to derive an interpolation
approach to solving a system of equation of second type as follows
\[
f\,:\,\mathbb{C}^{1\times n}\rightarrow\mathbb{C}
\]
such that 
\[
S\,:=\left\{ \left(\mathbf{x}=\mathbf{A}\left[i,:\right],\,f\left(\mathbf{A}\left[i,:\right]\right)=b_{i}\right)\,:\,i\in\left[0,n\right)\cap\mathbb{Z}\right\} ,
\]
where $\mathbf{A}\in$ GL$_{n}\left(\mathbb{C}\right)$ we have 
\begin{equation}
f\left(\mathbf{x}\right)=\prod_{0\le i<n}b_{i}^{\left(\underset{_{0\le j\ne i<n}}{\bigcirc}\left(\mathbf{x}-\mathbf{A}\left[j,:\right]^{\top}\right)\right)^{\top}\left[\text{diag}\left(\underset{_{0\le j\ne i<n}}{\bigcirc}\left(\mathbf{A}\left[i,:\right]-\mathbf{A}\left[j,:\right]\right)\right)^{+}\left(\frac{\mathbf{w}^{\circ^{0}}}{n}+\underset{0<k<n}{\sum}\gamma_{k}\mathbf{w}^{\circ^{k}}\right)\right]}
\end{equation}
The log-linearity of $f\left(\mathbf{x}\right)$ prescribed for all
$c\in\mathbb{C}$ and $\mathbf{u},\mathbf{v}\in\mathbb{C}^{1\times n}$
by 
\begin{equation}
\begin{array}{ccc}
f\left(c\,\mathbf{u}\right) & = & \left(f\left(\mathbf{u}\right)\right)^{c}\\
f\left(\mathbf{u}+\mathbf{v}\right) & = & f\left(\mathbf{u}\right)\cdot f\left(\mathbf{v}\right)
\end{array},
\end{equation}
comes into play in the determination of the solution expressed by
\begin{equation}
f\left(\mathbf{x}\right)\mod\left(f\left(\mathbf{0}\right)-1\right)
\end{equation}
As concrete illustration we consider the case $n=2$, where 
\begin{equation}
f\left(\mathbf{x}\right)=b_{1}^{\left(\mathbf{x}-\mathbf{A}\left[0,:\right]^{\top}\right)^{\top}\left(\begin{array}{c}
\frac{\nicefrac{1}{2}+\gamma}{a_{10}-a_{00}}\\
\frac{\nicefrac{1}{2}+\gamma}{a_{11}-a_{01}}
\end{array}\right)}b_{0}^{\left(\mathbf{x}-\mathbf{A}\left[1,:\right]^{\top}\right)^{\top}\left(\begin{array}{c}
\frac{\nicefrac{1}{2}+\gamma}{a_{00}-a_{10}}\\
\frac{\nicefrac{1}{2}+\gamma}{a_{01}-a_{11}}
\end{array}\right)}
\end{equation}
The code setup\\
\begin{sageverbatim}
sage: # Loading the Hypermatrix package
sage: load('./Hypermatrix_Algebra_tst.sage')
sage: 
sage: # Initialization of the variables
sage: gamma=var('gamma'); HM(2,2,'a')
[[a00, a01], [a10, a11]]
sage: var_list('b',2); var_list('x',2)
[b0, b1]
[x0, x1]
sage: F=(\
....: b1^(HM(1,2,[x0-a00,x1-a01])*HM(2,1,[(1/2+gamma)/(a10-a00), (1/2-gamma)/(a11-a01)]))[0,0]*\
....: b0^(HM(1,2,[x0-a10,x1-a11])*HM(2,1,[(1/2+gamma)/(a00-a10), (1/2-gamma)/(a01-a11)]))[0,0])
sage: 
sage: # Solving for the value of gamma
sage: SlnF=solve(F.subs([x0==0,x1==0])-1,gamma)
sage: 
sage: # Initialization of the construction
sage: G=(F.subs(SlnF)).canonicalize_radical()

\end{sageverbatim}\\
Running the code yields the following constraint in the parameter
$\gamma$
\[
1=\frac{b_{1}^{\frac{2\,a_{00}a_{01}-a_{01}a_{10}-a_{00}a_{11}+2\,\left(a_{01}a_{10}-a_{00}a_{11}\right)\gamma-\left(2\,\left(a_{01}-a_{11}\right)\gamma+a_{01}-a_{11}\right)x_{0}+\left(2\,\left(a_{00}-a_{10}\right)\gamma-a_{00}+a_{10}\right)x_{1}}{2\,\left(a_{00}a_{01}-a_{01}a_{10}-\left(a_{00}-a_{10}\right)a_{11}\right)}}}{b_{0}^{\frac{a_{01}a_{10}+\left(a_{00}-2\,a_{10}\right)a_{11}+2\,\left(a_{01}a_{10}-a_{00}a_{11}\right)\gamma-\left(2\,\left(a_{01}-a_{11}\right)\gamma+a_{01}-a_{11}\right)x_{0}+\left(2\,\left(a_{00}-a_{10}\right)\gamma-a_{00}+a_{10}\right)x_{1}}{2\,\left(a_{00}a_{01}-a_{01}a_{10}-\left(a_{00}-a_{10}\right)a_{11}\right)}}}
\]

We describe here a derivation of the least square solution by the
method of square completion. Let $\mathbf{A}\left(z\right)\in\left(\mathbb{C}^{\mathbb{C}}\right)^{m\times n}$
and $\mathbf{x}\left(z\right)\in\left(\mathbb{C}^{\mathbb{C}}\right)^{n\times1}$
respectively denote coefficient and variable constructs of the system,
with entries given by 
\[
\mathbf{A}\left(z\right)\left[i,j\right]=a_{ij}\,z-\frac{b_{i}}{n}\ \text{ and }\ \mathbf{x}\left(z\right)\left[j\right]=0\,z+x_{j}\,z^{0}\;\forall\,\begin{cases}
\begin{array}{c}
0\le i<m\\
0\le j<n
\end{array}\end{cases}.
\]
The completion of square argument is based on the spectral decomposition
as follows 
\[
\text{arg}\min_{\mathbf{x}}\left\{ \text{CProd}_{\sum}\left(\text{CProd}_{\sum}\left(\mathbf{A}\left(z\right),\mathbf{x}\right)^{*},\text{CProd}_{\sum}\left(\mathbf{A}\left(z\right),\mathbf{x}\right)\right)\right\} =\text{arg}\min_{\mathbf{x}}\left\{ \left(\mathbf{A}\mathbf{x}\right)^{*}\left(\mathbf{A}\mathbf{x}\right)-\left(\mathbf{A}\mathbf{x}\right)^{*}\mathbf{b}-\mathbf{b}^{*}\left(\mathbf{A}\mathbf{x}\right)+\mathbf{b}^{*}\mathbf{b}\right\} 
\]
\begin{equation}
=\text{arg}\min_{\mathbf{x}}\left\{ \mathbf{x}^{*}\mathbf{A}^{*}\mathbf{A}\mathbf{x}-\left(\mathbf{A}^{*}\mathbf{A}\mathbf{x}\right)^{*}\left(\mathbf{A}^{+}\mathbf{b}\right)-\left(\mathbf{A}^{+}\mathbf{b}\right)^{*}\left(\mathbf{A}^{*}\mathbf{A}\mathbf{x}\right)+\mathbf{b}^{*}\mathbf{b}\right\} 
\end{equation}
where $\mathbf{A}^{+}$ denotes the pseudo-inverse of $\mathbf{A}$.
For notational convenience, let $\mathbf{v}=\sqrt{\text{diag}\left(\boldsymbol{\lambda}\right)}\mathbf{Q}\mathbf{A}^{+}\mathbf{b}$,
and let the spectral decomposition of $\mathbf{A}^{*}\mathbf{A}$
be expressed 
\begin{equation}
\left(\mathbf{A}^{*}\mathbf{A}\right)^{k}=\left(\sqrt{\text{diag}\left(\boldsymbol{\lambda}\right)}^{k}\mathbf{Q}\right)^{*}\left(\sqrt{\text{diag}\left(\boldsymbol{\lambda}\right)}^{k}\mathbf{Q}\right)\;\forall\;0\le k\le n
\end{equation}
then
\begin{equation}
\text{arg}\min_{\mathbf{x}}\left\{ \left(\sqrt{\text{diag}\left(\boldsymbol{\lambda}\right)}\mathbf{Q}\mathbf{x}\right)^{*}\left(\sqrt{\text{diag}\left(\boldsymbol{\lambda}\right)}\mathbf{Q}\mathbf{x}\right)-\left(\sqrt{\text{diag}\left(\boldsymbol{\lambda}\right)}\mathbf{Q}\mathbf{x}\right)^{*}\mathbf{v}-\mathbf{v}^{*}\left(\sqrt{\text{diag}\left(\boldsymbol{\lambda}\right)}\mathbf{Q}\mathbf{x}\right)+\mathbf{v}^{*}\mathbf{v}+\left(\mathbf{b}^{*}\mathbf{b}-\mathbf{v}^{*}\mathbf{v}\right)\right\} .
\end{equation}
\begin{equation}
\implies\mathbf{Q}^{*}\left(\sqrt{\text{diag}\left(\boldsymbol{\lambda}\right)}\right)^{+}\mathbf{v}=\text{arg}\min_{\mathbf{x}}\left\{ \left(\sqrt{\text{diag}\left(\boldsymbol{\lambda}\right)}\mathbf{Q}\mathbf{x}-\mathbf{v}\right)^{*}\left(\sqrt{\text{diag}\left(\boldsymbol{\lambda}\right)}\mathbf{Q}\mathbf{x}-\mathbf{v}\right)\right\} 
\end{equation}
\begin{equation}
\implies\mathbf{Q}^{*}\left(\sqrt{\text{diag}\left(\boldsymbol{\lambda}\right)}\right)^{+}\sqrt{\text{diag}\left(\boldsymbol{\lambda}\right)}\mathbf{Q}\mathbf{A}^{+}\mathbf{b}=\text{arg}\min_{\mathbf{x}}\left\{ \left(\sqrt{\text{diag}\left(\boldsymbol{\lambda}\right)}\mathbf{Q}\mathbf{x}-\mathbf{v}\right)^{*}\left(\sqrt{\text{diag}\left(\boldsymbol{\lambda}\right)}\mathbf{Q}\mathbf{x}-\mathbf{v}\right)\right\} 
\end{equation}
which expresses the least square solution.

We now discuss the least square solution associated with constraints
of the form
\begin{equation}
\mathbf{b}=\text{CProd}_{\sum}\left(z^{\circ\mathbf{A}},\mathbf{x}\right)\Longleftrightarrow\mathbf{1}_{m\times1}=\text{CProd}_{\prod}\left(\mathbf{A}\left(z\right),\mathbf{x}\right)
\end{equation}
where 
\begin{equation}
\mathbf{A}\left(z\right)\left[i,j\right]=\frac{z^{a_{ij}}}{\sqrt[n]{b_{i}}},\ \forall\:\begin{cases}
\begin{array}{c}
0\le i<m\\
0\le j<n
\end{array}\end{cases}.
\end{equation}
The least square solution is obtained by solving for 
\begin{equation}
\text{arg}\min_{\mathbf{x}}\left\{ \text{CProd}_{\sum}\left(\ln_{\circ}\left(\text{CProd}_{\prod}\left(\mathbf{A}\left(z\right),\mathbf{x}\right)\right)^{*},\,\ln_{\circ}\left(\text{CProd}_{\prod}\left(\mathbf{A}\left(z\right),\mathbf{x}\right)\right)\right)\right\} 
\end{equation}
The third type of system of equation is also results from setting
the \emph{combinator} to $\prod$ and expressed by constraints of
the form 
\begin{equation}
\mathbf{1}_{m\times1}=\mbox{CProd}_{\prod}\left(\mathbf{A}\left(z\right),\mathbf{x}\left(z\right)\right),
\end{equation}
where $\mathbf{A}\left(z\right)\in\left(\mathbb{C}^{\mathbb{C}}\right)^{m\times n}$
is called the base construct and $\mathbf{x}\left(z\right)$ of size
$n\times1$ denotes the variable construct having entries given by
\begin{equation}
\mathbf{A}\left(z\right)\left[i,j\right]=\frac{a_{ij}^{z}}{\sqrt[n]{b_{i}}}\ \text{ and }\ \mathbf{x}\left(z\right)\left[j\right]=0\,z+x_{j}z^{0},\quad\forall\,\begin{cases}
\begin{array}{c}
0\le i<m\\
0\le j<n
\end{array}\end{cases}.
\end{equation}
For example, the constructs
\begin{equation}
\mathbf{A}\left(z\right)=\left(\begin{array}{cc}
\frac{a_{00}^{z}}{\sqrt{b_{0}}} & \frac{a_{01}^{z}}{\sqrt{b_{0}}}\\
\frac{a_{10}^{z}}{\sqrt{b_{1}}} & \frac{a_{11}^{z}}{\sqrt{b_{1}}}
\end{array}\right)\ \text{ and }\ \mathbf{x}\left(z\right)=\left(\begin{array}{c}
0z+x_{0}z^{0}\\
0z+x_{1}z^{0}
\end{array}\right),
\end{equation}
express a system of two equations in the unknowns $x_{0}$, $x_{1}$
given by 
\begin{equation}
\mathbf{1}_{2\times1}=\left(\begin{array}{r}
\frac{a_{00}^{x_{0}}\cdot a_{01}^{x_{1}}}{b_{0}}\\
\frac{a_{10}^{x_{0}}\cdot a_{11}^{x_{1}}}{b_{1}}
\end{array}\right).
\end{equation}
The code setup above initialize the constructs\begin{sagesilent}
# Loading the Package
load('./Hypermatrix_Algebra_tst.sage')

# Initialization of the morphism variable
z=var('z')

# Initialization of the order and size parameters
od=2; sz=2

# Initialization of the constructs
A=HM(sz, sz, 'a').elementwise_exponent(z)
B=HM(sz, sz, 'b')

# Computing the product
C1=GProd([A,B], prod, [z])

# The right identity construct
rId=z*HM(od, sz, 'kronecker')
C2=GProd([A,rId], prod, [z])

# The left identity construct
lId=HM(od, sz, 'kronecker').elementwise_base_exponent(z)
C3=GProd([lId,A], prod, [z])
\end{sagesilent}
\begin{equation}
\mathbf{A}\left(z\right)=\sage{A.matrix()},\;\mathbf{B}\left(z\right)=\sage{B.matrix()},
\end{equation}
\begin{equation}
\mathbf{rId}\left(z\right)=\sage{rId.matrix()},\;\mathbf{lId}\left(z\right)=\sage{lId.matrix()}
\end{equation}
and computes 
\begin{equation}
\begin{array}{ccc}
\text{GProd}_{\prod}\left(\mathbf{A},\mathbf{B}\right) & = & \sage{C1.matrix()},\\
\\
\text{GProd}_{\prod}\left(\mathbf{A},\mathbf{rId}\right) & = & \sage{C2.matrix()},\\
\\
\text{GProd}_{\prod}\left(\mathbf{lId},\mathbf{A}\right) & = & \sage{C3.matrix()}.
\end{array}
\end{equation}
The code setup illustrates the fact that left identity elements differs
from right identity elements for CProd$_{\prod}$. The least square
constraints associated with systems of type $3$ are associated with
the 
\begin{equation}
\mathbf{b}=\text{GProd}_{\prod}\left(\mathbf{A}^{\circ z},\mathbf{x}\right)\Longleftrightarrow\mathbf{1}_{m\times1}=\text{GProd}_{\prod}\left(\mathbf{A}\left(z\right),\mathbf{x}\right)
\end{equation}
where 
\begin{equation}
\mathbf{A}\left(z\right)\left[i,j\right]=\frac{a_{ij}^{z}}{\sqrt[n]{b_{i}}},\ \forall\:\begin{cases}
\begin{array}{c}
0\le i<m\\
0\le j<n
\end{array}\end{cases}
\end{equation}
The log-least square solution is obtained by solving the minimization
problem 
\begin{equation}
\arg\min_{\mathbf{x}}\left\{ \text{GProd}_{\sum}\left(\ln_{\circ}\left(\text{GProd}_{\prod}\left(\mathbf{A}\left(z\right),\mathbf{x}\right)\right)^{*},\,\ln_{\circ}\left(\text{GProd}_{\prod}\left(\mathbf{A}\left(z\right),\mathbf{x}\right)\right)\right)\right\} 
\end{equation}

\section{A spectral theory for constructs.}

The theory of construct broaden the scope of matrix/hypermatrix spectra.
We make this point by describing a concrete illustration of the spectral
decomposition of a $2\times2$ construct. Let the composer $\mathcal{F}\,:\,\mathbb{C}^{\mathbb{C}}\times\mathbb{C}^{\mathbb{C}}\rightarrow\mathbb{C}^{\mathbb{C}}$
and combinator be respectively set to 
\[
\mathcal{F}\left(f\left(z\right),g\left(z\right)\right):=f\left(g\left(z\right)\right)\quad\text{ and }\quad\begin{array}{c}
\\
\text{Op}\\
^{0\le\textcolor{red}{j}<k}
\end{array}\,:=\sum_{0\le\textcolor{red}{j}<k}
\]
Consider the $2\times2$ construct
\[
\mathbf{U}\left(z\right)=\left(\begin{array}{cc}
\frac{e^{z}}{2} & \frac{\ln z}{2}\\
\frac{-e^{z}}{2} & \frac{\ln z}{2}
\end{array}\right)\text{ and }\mathbf{V}\left(z\right)=\left(\begin{array}{cc}
\ln z & \ln\left(-z\right)\\
e^{z} & e^{z}
\end{array}\right)
\]
which from a pseudo-inverse pair in the sense that they satisfy the
equality
\[
\text{GProd}_{\sum,\mathcal{F}}\left\{ \left(\begin{array}{cc}
\frac{e^{z}}{2} & \frac{\ln z}{2}\\
\frac{-e^{z}}{2} & \frac{\ln z}{2}
\end{array}\right),\left(\begin{array}{cc}
\ln z & \ln\left(-z\right)\\
e^{z} & e^{z}
\end{array}\right)\right\} =\left(\begin{array}{cc}
z & 0\\
0 & z
\end{array}\right).
\]
The construct spectral decomposition is expressed analogously to the
matrix spectral decomposition as product of the form 
\[
\text{GProd}_{\sum,\mathcal{F}}\left\{ \left(\begin{array}{cc}
\frac{e^{z}}{2} & \frac{\ln z}{2}\\
\frac{-e^{z}}{2} & \frac{\ln z}{2}
\end{array}\right),\text{GProd}_{\sum,\mathcal{F}}\left\{ \left(\begin{array}{cc}
\lambda_{0}\left(z\right) & 0\\
0 & \lambda_{1}\left(z\right)
\end{array}\right),\left(\begin{array}{cc}
\ln z & \ln\left(-z\right)\\
e^{z} & e^{z}
\end{array}\right)\right\} \right\} =
\]
\begin{equation}
\text{GProd}_{\sum,\mathcal{F}}\left\{ \left(\begin{array}{cc}
\frac{e^{z}}{2} & \frac{\ln z}{2}\\
\frac{-e^{z}}{2} & \frac{\ln z}{2}
\end{array}\right),\left(\begin{array}{cc}
\lambda_{0}\left(\ln z\right) & \lambda_{0}\left(\ln\left(-z\right)\right)\\
\lambda_{1}\left(e^{z}\right) & \lambda_{1}\left(e^{z}\right)
\end{array}\right)\right\} \label{Spectral decompotion 1}
\end{equation}
or alternatively 
\[
\text{GProd}_{\sum,\mathcal{F}}\left\{ \text{GProd}_{\sum,\mathcal{F}}\left\{ \left(\begin{array}{cc}
\frac{e^{z}}{2} & \frac{\ln z}{2}\\
\frac{-e^{z}}{2} & \frac{\ln z}{2}
\end{array}\right),\left(\begin{array}{cc}
\lambda_{0}\left(z\right) & 0\\
0 & \lambda_{1}\left(z\right)
\end{array}\right)\right\} ,\left(\begin{array}{cc}
\ln z & \ln\left(-z\right)\\
e^{z} & e^{z}
\end{array}\right)\right\} =
\]
\begin{equation}
\text{GProd}_{\sum,\mathcal{F}}\left\{ \left(\begin{array}{cc}
\frac{e^{\lambda_{0}\left(z\right)}}{2} & \frac{\ln\left(\lambda_{1}\left(z\right)\right)}{2}\\
\frac{-e^{\lambda_{0}\left(z\right)}}{2} & \frac{\ln\left(\lambda_{1}\left(z\right)\right)}{2}
\end{array}\right),\left(\begin{array}{cc}
\ln z & \ln\left(-z\right)\\
e^{z} & e^{z}
\end{array}\right)\right\} .\label{Spectral decomposition 2}
\end{equation}
Both Eq. ( \ref{Spectral decompotion 1} ) and Eq. ( \ref{Spectral decomposition 2}
) express the spectral decomposition of the $2\times2$ construct
\begin{equation}
2^{-1}\left(\begin{array}{cc}
e^{\lambda_{0}\left(\ln z\right)}+\ln\left(\lambda_{1}\left(e^{z}\right)\right) & e^{\lambda_{0}\left(\ln\left(-z\right)\right)}+\ln\left(\lambda_{1}\left(e^{z}\right)\right)\\
-e^{\lambda_{0}\left(\ln z\right)}+\ln\left(\lambda_{1}\left(e^{z}\right)\right) & -e^{\lambda_{0}\left(\ln\left(-z\right)\right)}+\ln\left(\lambda_{1}\left(e^{z}\right)\right)
\end{array}\right)\text{ for }\lambda_{1}\left(z\right),\lambda_{2}\left(z\right)\subset\mathbb{C}^{\mathbb{C}}.
\end{equation}
At $z=1$, the spectral decomposition of a given construct $\mathbf{A}\left(z\right)\in\left(\mathbb{C}^{\mathbb{C}}\right)^{n\times n}$
expresses the spectral decomposition of $\mathbf{A}\left(1\right)$
( assuming that $\mathbf{A}\left(1\right)$ is diagonalizable ) if
every entries of $\mathbf{A}\left(z\right)$ is a polynomial of degree
at most one in the morphism variable $z$ and $\mathbf{A}\left(0\right)=\mathbf{0}_{n\times n}$.
The spectra of constructs therefore generalizes the spectra of matrices
and in so doing reveals some subtle details of the matrix spectra.
For instance the construct spectra brings to light the importance
of commutativity in the formulation of eigenvalue-eigenvector constraints.
In particular, the construct formulation eigenvalue-eigenvector constraints
reveals a natural duality relating the defined composer $\mathcal{F}$
and its dual $\mathcal{G}\,:\,\mathbb{C}^{\mathbb{C}}\times\mathbb{C}^{\mathbb{C}}\rightarrow\mathbb{C}^{\mathbb{C}}$
such that for any $f,\,g\,\in\,\mathbb{C}^{\mathbb{C}}$ such that
\[
\mathcal{G}\left(f\left(z\right),g\left(z\right)\right)=g\left(f\left(z\right)\right).
\]
Note that
\[
\mathcal{F}\left(f\left(z\right),g\left(z\right)\right)-\mathcal{G}\left(f\left(z\right),g\left(z\right)\right)\quad\text{ and }\quad\mathcal{F}\left(f\left(z\right),g\left(z\right)\right)\left(\mathcal{G}\left(f\left(z\right),g\left(z\right)\right)\right)^{-1}
\]
both express commutation relations which are central to the study
of Lie algebras. In the context of the $2\times2$ illustration, the
eigenvalues-eigenvector equation are 
\[
\text{GProd}_{\sum,\mathcal{F}}\left\{ 2^{-1}\left(\begin{array}{cc}
e^{\lambda_{0}\left(\ln z\right)}+\ln\left(\lambda_{1}\left(e^{z}\right)\right) & e^{\lambda_{0}\left(\ln\left(-z\right)\right)}+\ln\left(\lambda_{1}\left(e^{z}\right)\right)\\
-e^{\lambda_{0}\left(\ln z\right)}+\ln\left(\lambda_{1}\left(e^{z}\right)\right) & -e^{\lambda_{0}\left(\ln\left(-z\right)\right)}+\ln\left(\lambda_{1}\left(e^{z}\right)\right)
\end{array}\right),\,\left(\begin{array}{c}
\frac{e^{z}}{2}\\
\frac{-e^{z}}{2}
\end{array}\right)\right\} =
\]
\begin{equation}
\text{GProd}_{\sum,\mathcal{G}}\left\{ \left(\begin{array}{cc}
\lambda_{0}\left(z\right) & 0\\
0 & \lambda_{0}\left(z\right)
\end{array}\right),\left(\begin{array}{c}
\frac{e^{z}}{2}\\
\frac{-e^{z}}{2}
\end{array}\right)\right\} 
\end{equation}
and 
\[
\text{GProd}_{\sum,\mathcal{F}}\left\{ 2^{-1}\left(\begin{array}{cc}
e^{\lambda_{0}\left(\ln z\right)}+\ln\left(\lambda_{1}\left(e^{z}\right)\right) & e^{\lambda_{0}\left(\ln\left(-z\right)\right)}+\ln\left(\lambda_{1}\left(e^{z}\right)\right)\\
-e^{\lambda_{0}\left(\ln z\right)}+\ln\left(\lambda_{1}\left(e^{z}\right)\right) & -e^{\lambda_{0}\left(\ln\left(-z\right)\right)}+\ln\left(\lambda_{1}\left(e^{z}\right)\right)
\end{array}\right),\,\left(\begin{array}{c}
\frac{\ln z}{2}\\
\frac{\ln z}{2}
\end{array}\right)\right\} =
\]
\begin{equation}
\text{GProd}_{\sum,\mathcal{G}}\left\{ \left(\begin{array}{cc}
\lambda_{1}\left(z\right) & 0\\
0 & \lambda_{1}\left(z\right)
\end{array}\right),\left(\begin{array}{c}
\frac{\ln z}{2}\\
\frac{\ln z}{2}
\end{array}\right)\right\} .
\end{equation}
More generally the spectral decomposition of a construct $\mathbf{A}\left(z\right)\in\left(\mathbb{C}^{\mathbb{C}}\right)^{n\times n}$
is defined in terms of a pair of constructs $\mathbf{U}\left(z\right)\in\left(\mathbb{C}^{\mathbb{C}}\right)^{n\times n}$
and $\mathbf{V}\left(z\right)\in\left(\mathbb{C}^{\mathbb{C}}\right)^{n\times n}$
subject to
\begin{equation}
\text{GProd}_{\sum,\mathcal{F}}\left(\mathbf{U}\left(z\right),\mathbf{V}\left(z\right)\right)=z\,\mathbf{I}_{n}
\end{equation}
and the equality
\[
\mathbf{A}\left(z\right)=\text{GProd}_{\sum,\mathcal{F}}\left(\mathbf{U}\left(z\right),\text{GProd}_{\sum,\mathcal{F}}\left(\text{diag}\left(\begin{array}{c}
\lambda_{0}\left(z\right)\\
\vdots\\
\lambda_{n-1}\left(z\right)
\end{array}\right),\mathbf{V}\left(z\right)\right)\right)=
\]
\begin{equation}
\text{GProd}_{\sum,\mathcal{F}}\left(\text{GProd}_{\sum,\mathcal{F}}\left(\mathbf{U}\left(z\right),\text{diag}\left(\begin{array}{c}
\lambda_{0}\left(z\right)\\
\vdots\\
\lambda_{n-1}\left(z\right)
\end{array}\right)\right),\mathbf{V}\left(z\right)\right).
\end{equation}
for $\left\{ \lambda_{i}\left(z\right)\right\} _{0\le i<n}\subset\mathbb{C}^{\mathbb{C}}$.
Recall that the eigenvector-eigenvalue for an $n\times n$ matrix
$\mathbf{A}$ originate from the matrix equation 
\[
\mathbf{A}\cdot\mathbf{V}=\mathbf{V}\cdot\text{diag}\left(\begin{array}{c}
\lambda_{0}\\
\vdots\\
\lambda_{n-1}
\end{array}\right).
\]
The construct analog of the equation above is 
\begin{equation}
\text{GProd}_{\sum,\mathcal{F}}\left(\mathbf{A}\left(z\right),\,\mathbf{V}\left(z\right)\right)=\text{GProd}_{\sum,\mathcal{F}}\left(\mathbf{V}\left(z\right),\,\text{diag}\left(\begin{array}{c}
\lambda_{0}\left(z\right)\\
\vdots\\
\lambda_{n-1}\left(z\right)
\end{array}\right)\right).
\end{equation}
We express column-wise the eigenvalue-eigenvector equation by using
the primal and dual composer $\mathcal{F}$ and $\mathcal{G}$ as
follows 
\begin{equation}
\text{GProd}_{\sum,\mathcal{F}}\left(\mathbf{A}\left(z\right),\mathbf{V}\left(z\right)\left[:,i\right]\right)=\text{GProd}_{\sum,\mathcal{G}}\left(\text{diag}\left(\begin{array}{c}
\lambda_{i}\left(z\right)\\
\vdots\\
\lambda_{i}\left(z\right)
\end{array}\right),\mathbf{V}\left(z\right)\left[:,i\right]\right).
\end{equation}
Similarly to the matrix case not all constructs are diagonalizable
consequently determining the spectra of second order construct when
the combinator is set to 
\[
\begin{array}{c}
\\
\text{Op}\\
^{0\le\textcolor{red}{j}<k}
\end{array}\,:=\sum_{0\le\textcolor{red}{j}<k}
\]
amounts to solving constraints of the form 
\begin{equation}
\begin{array}{c}
\text{GProd}_{\sum,\mathcal{F}}\left(\mathbf{A}\left(z\right),\mathbf{v}\left(z\right)\right)=\text{GProd}_{\sum,\mathcal{G}}\left(\lambda\left(z\right)\mathbf{I}_{n},\mathbf{v}\left(z\right)\right)\\
\text{or}\\
\text{GProd}_{\sum,\mathcal{F}}\left(\mathbf{A}\left(z\right),\mathbf{v}\left(z\right)\right)=\text{GProd}_{\sum,\mathcal{F}}\left(\lambda\left(z\right)\mathbf{I}_{n},\mathbf{v}\left(z\right)\right)
\end{array}.
\end{equation}
Similarly, determining the spectra of second order construct when
the combinator is set to 
\[
\begin{array}{c}
\\
\text{Op}\\
^{0\le\textcolor{red}{j}<k}
\end{array}\,:=\prod_{0\le\textcolor{red}{j}<k}
\]
amounts to solving constraints of the form 
\begin{equation}
\begin{array}{c}
\text{GProd}_{\prod,\mathcal{F}}\left(\mathbf{A}\left(z\right),\mathbf{v}\left(z\right)\right)=\text{GProd}_{\prod,\mathcal{G}}\left(\left(\lambda\left(z\right)\right)^{\circ^{\mathbf{I}_{n}}},\mathbf{v}\left(z\right)\right)\\
\text{or}\\
\text{GProd}_{\prod,\mathcal{F}}\left(\mathbf{A}\left(z\right),\mathbf{v}\left(z\right)\right)=\text{GProd}_{\prod,\mathcal{F}}\left(\left(\lambda\left(z\right)\right)^{\circ^{\mathbf{I}_{n}}},\mathbf{v}\left(z\right)\right)
\end{array}.
\end{equation}
Similar construct eigenvalue-eigenvector constraints arise from tropical
linear algebra\cite{73605}.

\bibliographystyle{amsalpha}
\bibliography{mybib}

\bigskip{}

\parbox[t]{0.55\columnwidth}{%
\noun{Department of Applied Mathematics and Statistics}

\noun{Johns Hopkins University}

\noun{Baltimore, MD 21218}

E-mail: \texttt{egnang1@jhu.edu}%
}%
\parbox[t]{0.44\columnwidth}{%
\noun{Department of Epidemiology}

\noun{Johns Hopkins Bloomberg School of Public Health}

\noun{Baltimore, MD 21205}

E-mail: \texttt{jgnang1@jhu.edu}%
}
\end{document}